\documentclass[a4paper]{article}%

\usepackage{a4wide}
\usepackage{amsmath}
\usepackage{amsfonts}
\usepackage{amssymb}
\usepackage{authblk}
\usepackage{amsthm}
\usepackage{bbm}
\usepackage{graphicx}
\usepackage{caption}
\usepackage{mathrsfs}
\usepackage{tabularx}
\usepackage{enumitem}
\usepackage{subcaption}
\usepackage{xspace}
\usepackage{xcolor}
\usepackage{booktabs}
\usepackage[english]{babel}
\usepackage{marvosym}
\usepackage{cite}
\usepackage{hyperref} 
\usepackage{ragged2e}

\newtheorem{theorem}{Theorem}

\newtheorem{condition}[theorem]{Condition}
\newtheorem{corollary}[theorem]{Corollary}

\newtheorem{lemma}[theorem]{Lemma}

\newtheorem{proposition}[theorem]{Proposition}

\setenumerate{label={\normalfont{(\roman*)}}} 

\newcommand{\sss}[1]{\scriptscriptstyle{#1}}
\newcommand{\inprob}{\overset{\sss{\Prob}}{\longrightarrow}}
\newcommand{\Prob}{\mathbb{P}}
\newcommand*{\swap}[2]{#2#1}

\title{Hierarchical configuration model}
\author{Remco van der Hofstad, Johan S.H. van Leeuwaarden, Clara Stegehuis}
\affil{Eindhoven University of Technology, Department of Mathematics and Computer Science, P.O. Box 513, 5600 MB Eindhoven, The Netherlands}

\begin{document}
	\maketitle
\begin{abstract}
We introduce a class of random graphs with a community structure, which we call the \textit{hierarchical configuration model}. On the inter-community level, the graph is a configuration model, and on the intra-community level, every vertex in the configuration model is replaced by a community: i.e., a small graph. These communities may have any shape, as long as they are connected. For these hierarchical graphs, we find the size of the largest component, the degree distribution and the clustering coefficient.
 Furthermore, we determine the conditions under which a giant percolation cluster exists, and find its size. 
\end{abstract}

\section{Introduction and model}\label{sec:intr}

A characteristic feature of many real-world complex networks is that the degree distribution obeys a power law. A popular model for such power-law networks is the \emph{configuration model}, a random graph with a prescribed degree distribution~\cite{bollobas2001}. A major shortcoming of this model, however, is that it is locally tree-like \--- it contains only a few short cycles and the graph next to most vertices is a tree \--- while a prominent feature of complex networks is that they often have a community structure~\cite{newman2002a}. The communities are highly connected and contain many short cycles, while edges between different communities are more scarce. Therefore, the configuration model is not a realistic model to study networks with a community structure. 

Several other random graph models have been proposed to include community structures or short cycles. For example, in~\cite{ball2009, coupechoux2014, trapman2007} communities are introduced in the form of households, i.e., complete graphs of a certain size. The random graph then remains a configuration model on the macroscopic level, while on the microscopic level each vertex of the graph can be replaced by a household. The households introduce a community structure in the graph, which creates random graphs with not only a prescribed degree distribution but also a tunable clustering coefficient.
Another way to incorporate short cycles in the configuration model is to introduce triangles~\cite{newman2009}. In this case, besides the vertex degrees, also the number of triangles each vertex belongs to is prescribed. Then triangles are formed by joining three nodes at random, and regular edges are formed as in the configuration model. The model was extended in~\cite{karrer2010} to include arbitrary subgraphs. 
Yet another method to include clustering in random graph models is to use \emph{random intersection graphs}~\cite{deijfen2009}, which allows one to prescribe the exponent of the power-law degree sequence and the amount of clustering. 
In~\cite{sah2014}, an algorithm is developed for creating a model that matches the community structure in real-world networks. This algorithm first randomizes the edges between different communities according to a configuration model, and then it randomizes the edges inside the communities. This creates a graph with a similar community structure and degree distribution as the real-world network. However, this model was only studied through extensive simulations; the analytical properties of this model were not studied in~\cite{sah2014}. 

In this paper, we introduce the hierarchical configuration model, a random graph model that can describe networks with an arbitrary community structure. This model has a hierarchical structure that consists of two levels. The macroscopic level consists of the connections between communities, and the microscopic level describes the connections inside communities.
Like in~\cite{ball2009} and~\cite{coupechoux2014}, we study a random graph which is a configuration model on the macroscopic level, and then add communities on the microscopic level. In real-world networks, however, communities do not have to be complete graphs, especially when the communities are large, as is typically observed through community detection algorithms~\cite{guimera2003}. 
We generalize the setting of~\cite{ball2009} and~\cite{coupechoux2014} to a configuration model in which a vertex can be replaced by any \textit{small} graph. This generalization makes it possible to apply the hierarchical configuration model to real-world data sets. When the community structure of a real-world network is detected by an algorithm, the hierarchical configuration model is able to produce random graphs that have a similar community structure. 
Furthermore, due to the general community structure of the hierarchical configuration model, several existing random graph models turn out to be special cases.  
The model developed in~\cite{house2010} is similar to our model, but in contrast to the hierarchical configuration model, it only allows for a finite number of different communities, and all communities have to be of constant size.
The advantage of the hierarchical configuration model is that it is quite flexible in its \textit{local} structure, yet it is still analytically tractable due to its mesoscopic locally tree-like structure. In~\cite{stegehuis2015, stegehuis2016}, we have further studied how this model fits real-world networks, the conclusion being that our model fits quite well. This is an important step to come to more realistic random  graph models for real-world networks.

This paper is organized as follows.
In Section~\ref{sec:model} we define the hierarchical configuration model. Section~\ref{sec:gen} presents several analytical results for the hierarchical configuration model, including the condition for a giant component to emerge, the degree distribution and the clustering coefficient. In Section~\ref{sec:perc} we study bond percolation on the hierarchical configuration model.
Section~\ref{sec:exex} describes examples of graph models in the literature that fit into our general framework. Then we show in Section~\ref{sec:sty} how some stylized community structures affect percolation. Finally, we present some conclusions in Section~\ref{sec:disc}. 
 
 \paragraph*{Notation.}\label{sec:notation}
We use $\overset{d}\longrightarrow$ for convergence in distribution, and $\inprob $ for convergence in probability. We say that a sequence of events $(\mathcal{E}_n)_{n\geq 1}$ happens with high probability (w.h.p.) if $\lim_{n\to\infty}\Prob(\mathcal{E}_n)=1$. Furthermore, we write $f(n)=o(g(n))$ if $\lim_{n\to\infty}f(n)/g(n)=0$, and $f(n)=O(g(n))$ if $|f(n)|/g(n)$ is uniformly bounded, where $(g(n))_{n\geq 1}$ is nonnegative. We say that $X_n=O_{\sss{\Prob}}(b_n)$ for a sequence of random variables $(X_n)_{n\geq 1}$ if $|X_n|/b_n$ is a tight sequence of random variables, and $X_n=o_{\sss{\Prob}}(b_n)$ if $X_n/b_n\inprob 0$. Table~\ref{tab:addlabel}, at the end of this paper, contains a list of symbols that are used frequently throughout the manuscript.

\subsection{Hierarchical configuration model}\label{sec:model}

We now describe the random graph model that we introduce and study in this paper. 
Consider a random graph $G$ with $n$ communities. A community $H$ is represented by $H=(F,(d_v^{\sss{(b)}})_{v\in V})$, where $F=(V_F,E_F)$ is a simple, connected graph, and $d_v^{\sss{(b)}}$ is the number of edges from $v\in V_F$ to other communities. Thus $(d_v^{\sss{(b)}})_{v\in V_F}$ describes the degrees between the communities. We call $d_v^{\sss{(b)}}$ the \emph{inter-community degree} of a vertex. A vertex inside a community also has an \emph{intra-community degree} $d_v^{\sss{(c)}}$: the number of edges from that vertex to other vertices in the same community. The sum of the inside- and the inter-community degree of the vertex is the degree of the vertex, i.e., $d_v=d_v^{\sss{(b)}}+d_v^{\sss{(c)}}$. Let $d_H=\sum_{v\in V_F}d_v^{\sss{(b)}}$ be the total number of edges out of community $H$. On the macroscopic level, $G$ is a configuration model with degrees $d_H$. Let this macroscopic configuration model be denoted by $\phi(G)$.

Let $H_n=(F_n,\boldsymbol{d}_n)$ denote a uniformly chosen community in $[n]=\{1,2,\dots,n\}$. Furthermore, denote the number of communities of type $H$ in a graph with $n$ communities by $n^{\sss{(n)}}_H$. Then $n^{\sss{(n)}}_H/n$ is the fraction of communities that are of type $H$.
Let $D_n$ be the number of outgoing edges from a uniformly chosen community, i.e., $D_n=d_{H_n}$. Let the size of community $i$ be denoted by $s_i$, and the size of a uniformly chosen community in $[n]$ by $S_n\overset{d}{=}|F_n|$. Then the total number of vertices in the graph is $N=\sum_{i=1}^ns_i=n\mathbb{E}[S_n]$. We assume that the following conditions hold:

\begin{condition}[Community regularity]\label{cond:graph}
	\leavevmode
	\begin{enumerate}
		\item\label{cond:pn}
		$P_n(H)=n^{\sss{(n)}}_H/n\inprob P(H)$, where $P(H)$ is a probability distribution,
		\item\label{cond:S}
		$
		\lim_{n\to\infty}\mathbb{E}[S_n]=\mathbb{E}[S],$
	\end{enumerate}
	for some random variable $S$ with $\mathbb{E}[S]<\infty$.
\end{condition}

\begin{condition}[Intercommunity connectivity]\label{cond:size}
\leavevmode
\begin{enumerate}
\item \label{enum:ED}
$
\lim_{n\to\infty}\mathbb{E}[D_n]=\mathbb{E}[D],
$
\item \label{enum:PD}
$\Prob(D=2)<1$,
\end{enumerate}
for some random variable $D$ with $\mathbb{E}[D]<\infty$.
\end{condition}
Condition~\ref{cond:graph}\ref{cond:pn} implies $(F_n,\boldsymbol{d}_n)\overset{d}\longrightarrow (F,\boldsymbol{d})$, $D_n\overset{d}{\longrightarrow}D$ and $S_n\overset{d}{\longrightarrow}S$, so that $S$ and $D$ are the asymptotic community size distribution and community inter-community degree distribution, respectively. Define 
\begin{align}
p_{k,s}^{\sss{(n)}}&=\sum_{H=(F,\boldsymbol{d}):|F|=s,d_H=k}P_n(H),\\
p_{k,s}&=\sum_{H=(F,\boldsymbol{d}):|F|=s,d_H=k}P(H),
\end{align}
as the probabilities that a uniformly chosen community has size $s$ and inter-community degree $k$, for finite $n$ and $n\to\infty$, respectively. Then Condition~\ref{cond:graph} implies that $p_{k,s}^{\sss{(n)}}\to p_{k,s}$ for every $(k,s)$.

We can think of $P_n(H)$ as the probability that a uniformly chosen community has a certain shape.
In a data set we can approximate $P_n(H)$ and use the hierarchical configuration model in the following way.
Suppose a community detection algorithm gives the empirical distribution of the community shapes $P_n(H)$. Now we construct a random graph in the way that was described above. The probability that a certain community is of shape $H$ is $P_n(H)$. We condition on the total inter-community degree to be even so that edges between communities can be formed as in a configuration model. This results in a graph with roughly the same degree sequence as the original graph. Additionally, the community structure in the random graph is the same as in the original graph. 
This construction preserves more of the microscopic features of the original graph than a standard configuration model with the same degree sequence as the original graph. It also shows the necessity of extending the work of~\cite{coupechoux2014, trapman2007} to go beyond the assumption that communities are complete graphs, because communities in real-world networks can be non-complete. Using this construction, the hierarchical configuration model can match the community structure in many complex networks~\cite{stegehuis2015, stegehuis2016}.

\section{Model properties}\label{sec:gen}

For a connected component of $G$, we can either count the number of communities in the component, or the number of vertices in it. We denote the number of communities in a connected component $\mathscr{C}$ by $v(\mathscr{C}^{\sss{\text{H}}})$, and the number of communities with inter-community degree $k$ by $v_k(\mathscr{C}^{\sss{\text{H}}})$. The number of vertices in component $\mathscr{C}$ is denoted by $v(\mathscr{C})$. Let $\mathscr{C}_{\textup{max}}$  and $\mathscr{C}_2$ be the largest and second largest components of $G$, respectively, so that
	\begin{equation}
	v(\mathscr{C}_{\max})=\max_{u\in [N]}v(\mathscr{C}(u)),
	\end{equation}
	where $\mathscr{C}(u)$ denotes the component of vertex $u$.
Furthermore, define $\nu_D$ as
\begin{equation}
\nu_D=\frac{\mathbb{E}[D(D-1)]}{\mathbb{E}[D]},
\end{equation}
where $D$ is the asymptotic community degree of Condition~\ref{cond:size}.
Let $p_k=\Prob(D=k)$ and let $g(x)=\sum_kp_kx^k$ be the probability generating function of $D$, and $g'(x)=\sum_kkp_kx^{k-1}$ its derivative.

\subsection{Giant component}\label{sec:giant}

In the standard configuration model, a giant component exists w.h.p. if $\nu_D>1$~\cite{janson2009, molloy1998, molloy1995}. In the hierarchical configuration model a similar statement holds:

\begin{theorem}\label{thm:size}
	
Let $G$ be a hierarchical configuration model satisfying Conditions\textup{~\ref{cond:graph}} and\textup{~\ref{cond:size}}.
Then, 
\begin{enumerate}
\item \label{enum:nularge}
If $\nu_D>1$, 
\begin{equation}\label{eq:compsize}
\frac{v(\mathscr{C}_{\textup{max}})}{N}\inprob\frac{\sum_{k,s}sp_{k,s}(1-\xi^k)}{\mathbb{E}[S]}>0,
\end{equation}
where $\xi$ is the unique solution in $[0,1)$ of $g'(\xi)=\xi\mathbb{E}[D]$. Furthermore, $v(\mathscr{C}_2)/N\inprob 0$.
\item \label{enum:nusmall}
If $\nu_D\leq1$, then $v(\mathscr{C}_{\textup{max}})/N\inprob 0$.
\end{enumerate}
\end{theorem}

\begin{proof}

Suppose $\nu_D>1$. 
By~\cite[Theorem 4.1]{hofstadcomplex2}, if Condition~\ref{cond:size} holds, $D_n\inprob D$ and $\nu_D>1$ in a standard configuration model, then w.h.p.  there will be one component with a positive fraction of the vertices as $n\to\infty$. 
Furthermore, the number of vertices in the largest component in a standard configuration model $v(\mathscr{C}^{\sss{\text{CM}}}_{\textup{max}})$ and the number of vertices of degree $k$ in its largest connected component, $v_k(\mathscr{C}^{\sss{\text{CM}}}_{\textup{max}})$ satisfy
\begin{align}
v(\mathscr{C}^{\sss{\text{CM}}}_{\textup{max}})/n&\inprob 1-g(\xi)>0,\label{eq:commsize}\\
v_k(\mathscr{C}^{\sss{\text{CM}}}_{\textup{max}})/n&\inprob p_k(1-\xi^k) \label{eq:commsizek}.
\end{align}
If $\nu_D\leq1$, then $v(\mathscr{C}^{\sss{\text{CM}}}_{\textup{max}})/n\inprob0$. Therefore, if Conditions~\ref{cond:graph} and~\ref{cond:size} hold and $\nu_D>1$ in the hierarchical configuration model, then there is a component with a positive fraction of the communities as $n\to\infty$. Hence, we need to prove that the largest hierarchical component is indeed a large component with size given by~\eqref{eq:compsize} if $\nu_D>1$, and that a small hierarchical component is also a small component of $G$.

We denote the number of communities in the largest hierarchical component with inter-community degree $k$ and size $s$ by $v_{k,s}(\mathscr{C}^{\sss{\text{H}}}_{\text{max}})$. Since $G$ is a configuration model on the community level,~\eqref{eq:commsize} and~\eqref{eq:commsizek} apply on the community level. Furthermore, given a community in the largest hierarchical component of inter-community degree $k$, its size is \emph{independent} of being in the largest hierarchical component. Moreover, $\sum_ksv_{k,s}(\mathscr{C}^{\sss{\text{H}}}_{\text{max}})/n\leq \sum_ksp_{k,s}^{\sss{(n)}}$. Therefore, by Condition~\ref{cond:graph}, the fraction of vertices in the largest hierarchical component satisfies
\begin{align}\label{eq:giantcomm}
\frac{v(\mathscr{C}_{\textup{max}})}{N}=\frac{\sum_{i\in \mathscr{C}^{\textup{H}}_{\textup{max}}}s_i}{\sum_is_i}=\frac{\sum_{k,s}n^{-1}sv_{k,s}(\mathscr{C}^{\sss{\text{H}}}_{\text{max}})}{n^{-1}\sum_is_i}\inprob\frac{\sum_{k,s}sp_{k,s}(1-\xi^k)}{\mathbb{E}[S]}>0.
\end{align}
The last inequality follows from Condition~\ref{cond:graph}\ref{cond:S} and the fact that $\xi\in[0,1)$ and $s\geq1$.  Now we need to prove that the largest hierarchical component indeed is the largest component of $G$. We show that a hierarchical component of size $o_\Prob(n)$ is w.h.p. a component of size $o_{\mathbb{P}}(N)$. Take a hierarchical component $\mathscr{C}$ which is not the largest hierarchical component, so that it is of size $o_{\mathbb{P}}(n)$. Then,
\begin{equation}\label{eq:Clb}
\begin{aligned}[b]
\frac{v(\mathscr{C})}{N}&=\frac{n^{-1}\sum_{k,s}sv_{k,s}(\mathscr{C}^{\sss{\text{H}}})}{\mathbb{E}[S_n]}=\frac{n^{-1}\sum_{s=1}^K\sum_{k}sv_{k,s}(\mathscr{C}^{\sss{\text{H}}})}{\mathbb{E}[S_n]}+\frac{n^{-1}\sum_{s>K}\sum_{k}sv_{k,s}(\mathscr{C}^{\sss{\text{H}}})}{\mathbb{E}[S_n]}\\
&\leq K\frac{n^{-1}\sum_{s=1}^K\sum_{k}v_{k,s}(\mathscr{C}^{\sss{\text{H}}})}{\mathbb{E}[S_n]}+\frac{\mathbb{E}[S_n\mathbbm{1}_{\{S_n>K\}}]}{\mathbb{E}[S_n]}\leq
K\frac{n^{-1}v(\mathscr{C}^{\sss{\text{H}}})}{\mathbb{E}[S_n]}+\frac{\mathbb{E}[S_n\mathbbm{1}_{\{S_n>K\}}]}{\mathbb{E}[S_n]}.
\end{aligned}
\end{equation}
First we take the limit for $n\to\infty$, and then we let $K\to\infty$. 
By~\cite{janson2009}, $v(\mathscr{C}^{\sss{\text{H}}})/n\inprob 0$, hence the first term tends to zero as $n\to\infty$. Furthermore, $\mathbb{E}[S_n\mathbbm{1}_{\{S_n>K\}}]\to\mathbb{E}[S\mathbbm{1}_{\{S>K\}}]$ as $n\to\infty$ by Condition~\ref{cond:graph}. By Condition~\ref{cond:graph}\ref{cond:S}, this tends to zero as $K\to\infty$. Thus, $v(\mathscr{C})/N\inprob 0$. Since~\eqref{eq:Clb} is uniform in $\mathscr{C}$, this proves that the largest hierarchical component is indeed the largest component of $G$.
This also proves~\ref{enum:nusmall}, since by~\cite[Theorem~4.1]{hofstadcomplex2}, if $\nu_D\leq 1$, $v(\mathscr{C}^{\sss{\text{H}}}_{\text{max}})=o_{\mathbb{P}}(n)$, so that $v(\mathscr{C}_{\textup{max}})=o_{\mathbb{P}}(N)$.
\end{proof}

We conclude that if Conditions~\ref{cond:graph} and~\ref{cond:size} hold and $\nu_D>1$, then a giant component exists in the hierarchical configuration model. Equation~\eqref{eq:compsize} gives the fraction of vertices in the largest component. The fraction of vertices in the giant component may be different from the fraction of communities in the giant hierarchical component. 
If the sizes and the inter-community degrees of the communities are independent, then the fraction of vertices in the largest component is equal to the fraction of communities in the largest hierarchical component.

\begin{corollary}\label{cor:indep}
Suppose that in the hierarchical configuration model $G$ satisfying Conditions\textup{~\ref{cond:graph}} and\textup{~\ref{cond:size}}, the size of the communities and the inter-community degrees of the communities are independent. Then, if $\nu_D>1$,
\begin{align}
\frac{v(\mathscr{C}_{\textup{max}})}{N}&\inprob 1-g(\xi),\label{eq:indepsize}\\
\frac{v(\mathscr{C}^{\textup{H}}_{\textup{max}})}{n}&\inprob 1-g(\xi)\label{eq:indepsize2},
\end{align}
where $\xi$ is the unique solution in $[0,1)$ of $g'(\xi)=\xi\mathbb{E}[D]$. Hence the fraction of vertices in the largest component is equal to the fraction of communities in the largest hierarchical component. If the size and the inter-community degrees are dependent, then this does not have to be true.
\end{corollary}
\begin{proof}
The equality in~\eqref{eq:indepsize2} is given by~\cite[Theorem~4.1]{hofstadcomplex2}. The equality in~\eqref{eq:indepsize} follows by substituting $p_{k,s}=p_kp_s$ in~\eqref{eq:compsize}, so that
\begin{equation}
\frac{v(\mathscr{C}_{\textup{max}})}{N}\inprob\frac{\sum_{s}sp_s\sum_kp_k(1-\xi^k)}{\mathbb{E}[S]}=\frac{\mathbb{E}[S](1-\sum_kp_k\xi^k)}{\mathbb{E}[S]}=1-g(\xi).
\end{equation}
To show that~\eqref{eq:indepsize} may not hold when the inter-community degrees and the sizes are dependent, consider the hierarchical configuration model with
\begin{equation}
p_{k,s}=\begin{cases}
\frac13&\text{if }(k,s)=(3,10),\\
\frac23&\text{if }(k,s)=(1,1).
\end{cases}
\end{equation}
Since $\nu_D=\tfrac65>1$, a giant component exists w.h.p. Furthermore, $\xi$ solves
\begin{equation}
\tfrac23+\xi^2=\tfrac53\xi,
\end{equation}
which has $\tfrac23$ as its only solution in $[0,1)$. Therefore, the fraction of communities in the largest component is given by $1-g(\frac23)=\frac{37}{81}$.
To find the fraction of vertices in the largest component, we use~\eqref{eq:compsize}, which gives 
\begin{equation}
\tfrac14(\tfrac23(1-\tfrac12)+10\tfrac13(1-(\tfrac12)^3))=\tfrac{13}{16}>\tfrac{37}{81}.
\end{equation}
Thus, the fraction of vertices in the largest component is larger than the fraction of communities in the largest component.
\end{proof}
If there is a difference between the fraction of communities and the fraction of vertices in the largest component, then this difference is caused by the dependence of the sizes and the inter-community degrees of the communities. A community with a large inter-community degree has a higher probability of being in the largest hierarchical component than a community with a small inter-community degree. In the example in the proof of Corollary~\ref{cor:indep}, the communities with large inter-community degrees are large communities. This causes the fraction of vertices in the largest component to be larger than the fraction of communities in the largest hierarchical component.

\subsection{Degree distribution}\label{sec:degree}

In the hierarchical configuration model, the macroscopic configuration model has a fixed degree sequence. The degree distribution of $G$ depends on the sizes and shapes of the communities. Let $n_k^{\sss{(H)}}$ denote the number of vertices in community $H$ with the sum of their intra-community degree and inter-community degree equal to $k$. Then the degree distribution of the total graph $G$ is described in Proposition~\ref{prop:deg}:
\begin{proposition}\label{prop:deg}
Let $G$ be a hierarchical configuration model such that Conditions\textup{~\ref{cond:graph}} and\textup{~\ref{cond:size}} hold. The asymptotic probability $\hat{p}_k$ that a randomly chosen vertex inside $G$ has degree $k$ satisfies
\begin{equation}
\hat{p}_k=\frac{\sum_{H} P(H)n_k^{\sss{(H)}}}{\mathbb{E}[S]},
\end{equation}
 as $n\to\infty$.
\end{proposition}
\begin{proof}
Consider a hierarchical configuration model $G$ on $n$ communities. Let $n^{\sss{(n)}}_{H}$ be the number of communities in $G$ of type $H$. The total number of vertices of degree $k$ is the sum of the number of degree $k$ vertices inside all communities, hence it equals $\sum_{H} n^{\sss{(n)}}_{H}n_k^{\sss{(H)}}$. Furthermore, $P_n(H)n^{\sss{(H)}}_k\leq P_n(H)s_H$, so that $\lim_{n\to\infty}\sum_HP_n(N)n_k^{\sss{(H)}}=\sum_HP(H)n_k^{\sss{(H)}}$ by Condition~\ref{cond:graph}. This gives
\begin{equation}
\hat{p}_k^{\sss{(n)}}=\frac{n^{-1}\sum_{H} n_{H}n^{\sss{(H)}}_k}{n^{-1}N}=\frac{\sum_H P_n(H)n^{\sss{(H)}}_k}{\mathbb{E}[S_n]}
\inprob\hat{p}_k,
\end{equation}
as $n\to\infty$.
\end{proof}

\subsubsection{Power-law shift in dense communities}\label{sec:pl}

Proposition~\ref{prop:deg} shows that the inter-community degree distribution and the degree distribution of the graph may be different. A case of special interest is when the degree distribution follows a power law. Let $F_{X_n}(k)$ be the empirical distribution function of $n$ observations. Then we say that $X_n$ follows a power law with exponent $\tau$ if $0<c_1<c_2$ and $k_n$ exist such that
\begin{align}
1-F_{X_n}(k)&\geq c_1k^{-\tau+1}\quad\quad \forall k\leq k_n,\label{eq:deflow}\\
1-F_{X_n}(k)&\leq c_2k^{-\tau+1}\quad\quad \forall k\label{eq:defup},
\end{align}
where $k_n\to\infty$ as $n\to\infty$. Note that~\eqref{eq:deflow} cannot be true for \textit{all} $k$ and $X_n=D_n$, since $c_1k^{-\tau+1}<0$, and $1-F_{X_n}(k)=0$ for $k>\max_{i\in[n]} d_i^{\sss{(b)}}$.

In the case of a power-law degree distribution, the degree distribution and the community size distribution can be related when the communities are dense enough. We call a community $H$ $(\eta, \varepsilon)$-dense if 
\begin{equation}\label{eq:edense}
\#\left\{v\in H: d_v^{\sss{(c)}}\geq \eta (s-1) \right\}\geq\varepsilon s.
\end{equation}
This condition states that at least a fraction of $\varepsilon$ of the vertices of community $H$ have edges to at least a fraction of $\eta$ other vertices in the same community. The special case in which $\varepsilon=\eta=1$ corresponds to complete graph communities. We now show that for dense communities, the power-law exponents of the community sizes and the degree distribution are related. For the special case of household communities, this relation was already observed in~\cite{trapman2007}.

\begin{proposition}\label{prop:plshift}
Let $G$ be a hierarchical configuration model such that Conditions\textup{~\ref{cond:graph}} and\textup{~\ref{cond:size}} hold. Suppose that there exists a $K\geq 0$ such that $d^{\sss{\textup{(b)}}}\leq Ks$ for all vertices, where $s$ is the community size. Furthermore, assume that there exist $\varepsilon,\eta>0$ such that every community of $G$ is $(\eta,\varepsilon)$-dense. Then, the community size distribution $S$ follows a power-law distribution with exponent $\tau'$ with $\tau'>2$ if and only if the degree distribution follows a power law with exponent $\tau=\tau'-1$ where $\tau>1$.
\end{proposition}

\begin{proof}

First, assume that $S_n$ obeys a power law with exponent $\tau'>2$, so that for some $0<b_1<b_2$ and $k_n$, $b_1k^{-\tau'+1}\leq1-F_{S_n}(k)$ for all $k\leq k_n$ and $1-F_{S_n}(k)\leq b_2k^{-\tau'+1}$ for all $k$. 
Then the cumulative distribution function of the degrees $\hat{D}_N$ in a hierarchical configuration model $G$ on $N$ vertices, $F_{\hat{D}_N}(k)$, satisfies
\begin{equation}\label{eq:pllow}
\begin{aligned}[b]
1-F_{\hat{D}_N}(k)&=\frac1N\sum_{i=1}^N\mathbbm{1}\{d_i\geq k\}
=\frac1N\sum_{i=1}^N\mathbbm{1}\{d_i^{\sss{\text{(c)}}}+d_i^{\sss{\text{(b)}}}
\geq k\}\geq\frac1N\sum_{i=1}^N\mathbbm{1}\{d_i^{\sss{\text{(c)}}}\geq k\}\\
 &\geq  \frac\varepsilon N\sum_{i=1}^N\mathbbm{1}\{\eta (s_i-1)\geq k\}
 =\frac{\varepsilon}{n\mathbb{E}[S_n]}\sum_{i=1}^ns_i\mathbbm{1}\{ s_i\geq \frac{k}{\eta}+1\}\\
& \geq \frac{\varepsilon}{n\mathbb{E}[S_n]}\sum_{i=1}^n\frac{k}{\eta}\mathbbm{1}\{ s_i\geq \frac{k}{\eta}+1\}
 =\frac{k\varepsilon}{\eta}\left(1-F_{S_n}(k/\eta+1)\right)\\
 &\geq\varepsilon c_1(\eta)k^{-\tau'+2}=\varepsilon c_1(\eta)k^{-\tau+1} \quad\quad\quad \forall k\leq k_n\eta,
\end{aligned}
\end{equation}
where $c_1(\eta)$ is a constant depending on $\eta$.
Because the communities are simple, $d^{\sss{(c)}}_i\leq s_i-1$ for all vertices. Hence,
\begin{equation}\label{eq:plup}
\begin{aligned}[b]
1-F_{\hat{D}_N}(k)&=\frac1N\sum_{i=1}^N\mathbbm{1}\{d_i^{\sss{\text{(c)}}}+d_i^{\sss{\text{(b)}}}\geq k\}
\leq\frac1N\sum_{i=1}^N\mathbbm{1}\{(s_i-1)+s_iK\geq k\}\\
&=\frac{1}{n\mathbb{E}[S_n]}\sum_{i=1}^{n}s_i\mathbbm{1}\{s_i\geq \frac{k+1}{K+1}\}=\frac{1}{\mathbb{E}[S_n]}\mathbb{E}[S_n\mathbbm{1}\{S_n\geq\frac{k+1}{K+1}\}]\\
&=\frac{1}{\mathbb{E}[S_n]}\left(\frac{k+1}{K+1}\left(1-F_{S_n}\left(\frac{k+1}{K+1}\right)\right)+\sum_{j\geq (k+1)/(K+1)}(1-F_{S_n}(j))\right)\\
&\leq c_2(K) k^{-\tau'+2}=c_2(K) k^{-\tau+1},
\end{aligned}
\end{equation}
for all $k$, with $c_2(K)$ a constant depending on $K$. Here $\sum_{j\geq (k+1)/(K+1)}(1-F_{S_n}(j))\leq c(K)k^{-\tau'+2}$ since $1-F_{S_n}(j)\leq b_2k^{-\tau'+1}$ and $\tau'>2$. Taking the limit of $N\to\infty$, equation~\eqref{eq:pllow} and~\eqref{eq:plup} imply that the degree distribution of $G$ follows a power law with exponent $\tau$, which proves the first part.

Now assume that the degree distribution $\hat{D}_N$ of $G$ obeys a power law with exponent $\tau$, so that $0<b_3<b_4$ and $k_N$ exist such that~\eqref{eq:deflow} and~\eqref{eq:defup} are satisfied. Then the community sizes are minimized if each community is a complete graph, and each vertex has inter-community degree exactly $Ks_i$. Then, $d_v=(s_i-1)+Ks_i$ or $s_i=(d_v+1)/(K+1)$. Hence, the cumulative distribution of $S_n$ in a graph with $n$ communities, $F_{S_n}$, satisfies
\begin{equation}\label{eq:pl2}
\begin{aligned}[b]
1-F_{S_n}(k)&=\frac1n\sum_{i=1}^n\mathbbm{1}\{s_i\geq k\}
\geq\frac{\mathbb{E}[S_n]}{N}\sum_{i=1}^N\frac{K+1}{d_i+1}\mathbbm{1}\left\{\frac{d_i+1}{K+1}\geq k\right\}\\
&= \frac{\mathbb{E}[S_n]}{N}\sum_{i=1}^N\frac{K+1}{d_i+1}\mathbbm{1}\{d_i\geq k(K+1)-1\}
={\mathbb{E}[S_n](K+1)}\mathbb{E}\left[\frac{\mathbbm{1}\{\hat{D}_N\geq k(K+1)-1\}}{\hat{D}_N+1}\right]\\
&={\mathbb{E}[S_n](K+1)}\mathbb{P}(\hat{D}_N>k(K+1)-1)\mathbb{E}[1/(\hat{D}_N+1)\mid\hat{D}_N\geq k(K+1)-1]\\
&\geq{\mathbb{E}[S_n](K+1)}\frac{\mathbb{P}(\hat{D}_N>k(K+1)-1)}{\mathbb{E}[\hat{D}_N+1\mid\hat{D}_N\geq k(K+1)-1]}\\
&={\mathbb{E}[S_n](K+1)}\frac{\mathbb{P}(\hat{D}_N>k(K+1)-1)^2}{\mathbb{E}[(\hat{D}_N+1)\mathbbm{1}\{\hat{D}_N+1\geq k(K+1)\}]},
\end{aligned}
\end{equation}
where we use that $\mathbb{E}[1/X]\geq 1/\mathbb{E}[X]$.
The term in the denominator satisfies
\begin{align}\label{eq:denom}
{\mathbb{E}[(\hat{D}_N+1)\mathbbm{1}\{\hat{D}_N+1\geq k(K+1)\}]}&\leq k(K+1)\mathbb{P}(\hat{D}_N+1>k(K+1))+\sum_{j\geq k(K+1)}\mathbb{P}(\hat{D}_N+1\geq j)\nonumber\\
&\leq C(K)b_4(k(K+1))^{-\tau+2},
\end{align}
for all $k$, where $C(K)$ is a constant depending on $K$. Combining~\eqref{eq:pl2} and~\eqref{eq:denom} yields
\begin{align}
1-F_{S_n}(k)&\geq {\mathbb{E}[S_n](K+1)}\frac{b_3(k(K+1))^{-2(\tau+1)}}{b_4(k(K+1))^{-\tau+2}C(K)}\geq c_3(K)k^{-\tau}=c_3(K)k^{-\tau'+1},
\end{align}
for all $k\leq k_N/(K+1)$, where $c_3(K)$ is a constant depending on $K$. We also have
\begin{equation}
\begin{aligned}[b]
1-F_{S_n}(k)&=\frac{\mathbb{E}[S_n]}{N}\sum_{i=1}^N\frac{1}{s_i}\mathbbm{1}\{s_i\geq k\}
\leq\frac{\mathbb{E}[S_n]}{N}\sum_{i=1}^N\frac{1}{k}\mathbbm{1}\{s_i\geq k\}\leq
\frac{\mathbb{E}[S_n]}{\varepsilon N}\sum_{i=1}^N\frac{1}{k}\mathbbm{1}\{d_i^{\sss{(c)}}\geq k\eta\}\\
&\leq
\frac{\mathbb{E}[S_n]}{\varepsilon N}\sum_{i=1}^N\frac{1}{k}\mathbbm{1}\{d_i^{\sss{(c)}}+d_i^{\sss{(b)}}\geq k\eta\}\leq\frac{1}{\varepsilon}c_4(\eta)k^{-\tau}=\frac{1}{\varepsilon}c_4(\eta)k^{-\tau'+1},
\end{aligned}
\end{equation}
 for all $k$. Taking the limit of $n\to\infty$ proves that $S$ has a power-law distribution with exponent $\tau'=\tau+1$.
\end{proof}
Proposition~\ref{prop:plshift} relates the degree distribution to the community size distribution of the hierarchical configuration model. Under a more restrictive assumption on the inter-community degrees of individual vertices, this also establishes a similar relation between the degree distribution of $G$ and the inter-community degree distribution of the communities in case of a power-law degree distribution:

\begin{corollary}\label{cor:pl}

Let $G$ be a hierarchical configuration model satisfying Conditions\textup{~\ref{cond:graph}} and\textup{~\ref{cond:size}}. Suppose that there exists a $K\geq 0$ such that $d_v^{\sss{(b)}}\leq K$ for all vertices $v$. Furthermore, assume that there exist $\varepsilon,\eta>0$ such that every community of $G$ is $(\eta,\varepsilon)$-dense. Then, the inter-community degree distribution of $G$ cannot have a power-law distribution with exponent smaller than $\tau+1$  if the degree distribution of $G$ follows a power-law distribution with exponent $\tau$, where $\tau>1$.

\end{corollary}

\begin{proof}

By Proposition~\ref{prop:plshift}, $S$ follows a power law with exponent $\tau+1$. 
Since $d_v^{\sss{(c)}}\leq K$, also $d_{H_i}\leq Ks_i$ for all communities $H_i$. Therefore, $D\preceq KS$, and hence $D$ cannot have a power-law distribution with exponent smaller than $\tau+1$. 
\end{proof}

Corollary~\ref{cor:pl} shows that the degrees between communities have smaller tails than the degrees of the graph. This is consistent with our view of communities being highly connected, while edges between communities are more scarce.
For example, if the degree distribution follows a power law with $\tau \in(2,3)$, then the inter-community degree distribution has an exponent that is at least 3. Therefore, the inter-community degree distribution has finite variance, whereas the degree distribution has infinite variance.  A property of configuration models with power-law exponents in (2,3) is that the probability of obtaining a simple graph vanishes. However, the inter-community connections in the hierarchical configuration model have exponent larger than 3, so that the probability of obtaining a simple graph remains uniformly positive. This suggests that the hierarchical configuration model is able to produce a random graph which has a positive probability of being simple, while the degree distribution has an exponent in (2,3). 

In a companion paper~\cite{stegehuis2015}, we study these power-law relations in more detail, and we show that in the case of communities that are less dense, different relations between $\tau$ and $\tau '$ may hold.

\subsubsection{The probability of obtaining a simple graph}\label{sec:simple}
In the standard configuration model, the probability of obtaining a simple graph converges to $\mathrm{e}^{-\nu/2-\nu^2/4}$ under the condition that $\mathbb{E}[D^2]<\infty$~\cite{hofstad2009}. In the hierarchical configuration model, the probability of obtaining a simple graph is largely dependent on the shapes of the communities. Since we have assumed that the communities are simple, only the inter-community edges can create self-loops and multiple edges.

Suppose that each vertex in a community has at most one half-edge to other communities, i.e.,$d_v^{\sss{(b)}}\in\{0,1\}$. A double edge in the macroscopic configuration model corresponds to a community where two vertices have an edge to the same other community. Since $d_v^{\sss{(b)}}\in\{0,1\}$, a double edge in the macroscopic configuration model cannot correspond to a double edge in the hierarchical configuration model. A self-loop in the macroscopic configuration model corresponds to an edge from one vertex $v$ inside a community to another vertex $w$ inside the same community. This self-loop in the macroscopic configuration model corresponds to a double edge in the hierarchical configuration model if an edge from $v$ to $w$ was already present in the community. Thus, when $d_v^{\sss{(b)}}\in\{0,1\}$ the probability that the macroscopic configuration model is simple is lower bounded by the probability that no self-loops exist in the macroscopic configuration model,
\begin{equation}\label{eq:simple}
\liminf_{n\to\infty}\Prob\left(G_n\text{ simple}\right)\geq \textrm{e}^{-\nu_D/2}.
\end{equation}
In the case of complete graph communities, every self-loop of the macroscopic configuration model corresponds to a double edge in the hierarchical configuration model. Therefore, equality holds when all communities are complete graphs.

\subsection{Clustering coefficient}\label{sec:clust}

The clustering coefficient $C$ of a random graph is defined as
\begin{equation}\label{eq:clustdef}
C=\frac{3\times\text{number of triangles}}{\text{number of connected triples}}.
\end{equation}
A connected triple is a vertex with edges to two different other vertices. Note that the order of the vertices to which the middle vertex is connected does not matter. The clustering coefficient can thus be interpreted as the proportion of connected triples that are triangles.
In the standard configuration model, the clustering coefficient tends to zero when $\mathbb{E}[D^2]<\infty$~\cite{newman2003}. Thus, in the hierarchical configuration model, we expect that the clustering is entirely caused by triangles inside communities. 

Another measure of clustering is the local clustering coefficient for vertices of degree $k$. This coefficient can be interpreted as the fraction of neighbors of degree $k$ vertices that are directly connected and is defined as
\begin{equation}
C_k=\frac{\text{number of pairs of connected neighbors of degree $k$ vertices}}{k(k-1)/2\times \text{number of degree } k\text{ vertices}}.
\end{equation}

As in Section~\ref{sec:degree}, let $n_k^{\sss{(H)}}$ denote the number of vertices in community $H$ with degree equal to $k$. Furthermore, let $P_v^{\sss{(H)}}$ denote the number of pairs of neighbors of a vertex $v\in V_H$ within community $H$ that are also neighbors of each other. We denote the clustering coefficient of community $H$ by $C_H$. Every vertex $v$ in community $H$ has $d_v^{\sss{(c)}}(d_v^{\sss{(c)}}-1)/2$ pairs of neighbors inside $H$. Hence, the total number of connected triples inside the community is given by $\sum_{v\in V_H}d_v^{\sss{(c)}}(d_v^{\sss{(c)}}-1)/2$. Then, by~\eqref{eq:clustdef}, 
\begin{equation}\label{eq:clustcom}
C_H=\frac{2\sum_{v\in V_H}P_v^{\sss{(H)}}}{\sum_{v\in V_H}d_v^{\sss{(c)}}(d_v^{\sss{(c)}}-1)}.
\end{equation}
Proposition~\ref{prop:clust} states that the clustering coefficient of the hierarchical configuration model can be written as a combination of the clustering coefficients inside communities. 
Let $\hat{D}$ denote the asymptotic degree as in Proposition~\ref{prop:deg}.

\begin{proposition}\label{prop:clust}

Let $G$ be a hierarchical configuration model satisfying Conditions\textup{~\ref{cond:graph}} and\textup{~\ref{cond:size}}, $\lim_{n\to\infty}\mathbb{E}[D_n^2]=\mathbb{E}[D^2]<\infty$ and $\lim_{N\to\infty}\mathbb{E}[\hat{D}_N^2]=\mathbb{E}[\hat{D}^2]<\infty$. Then the clustering coefficient $C^{\sss{(n)}}$ and average clustering coefficient for vertices of degree $k$, $C_k^{\sss{(n)}}$, satisfy
\begin{align}
C^{\sss{(n)}}&\inprob C:=\frac{2\sum_{H}\sum_{v\in V_H}P(H)C_H d_v^{\sss{(c)}}(d_v^{\sss{(c)}}-1)}{\sum_{H}\sum_{v\in V_H} P(H)d_v(d_v-1)},\label{eq:clust}\\
C_k^{\sss{(n)}}&\inprob C_k:=\frac{2\sum_H\sum_{v\in V_H:d_v=k}P(H)P_v^{\sss{(H)}}}{k(k-1)}\label{eq:clustk}.
\end{align}
\end{proposition}

\begin{proof}

In the hierarchical configuration model, the number of triples is deterministic. A vertex $v$ with degree $d_v$ has $d_v(d_v-1)/2$ pairs of neighbors. Thus, the total number of connected triples in $G$ is given by $\sum_{H}n^{\sss{(n)}}_{H}\sum_{v\in V_H}d_v(d_v-1)/2$, where $n^{\sss{(n)}}_{H}$ is the number of type $H$ communities. 

Triangles in $G$ can be formed in several ways. First of all, a triangle can be formed by three edges inside the same community. In this case the triangle in $G$ is formed by a triangle in one of its communities $H$. 
Another possibility to create a triangle is shown in Figure~\ref{fig:self1}. The black edges show edges inside a community, and the dashed edges are formed by edges in the macroscopic configuration model. This triangle is formed by two intra-community edges, and one edge of the macroscopic configuration model. Figure~\ref{fig:self2} shows that this inter-community edge is a self-loop of the macroscopic configuration model. One self-loop of the macroscopic configuration model can create multiple triangles; at most $s_i-2$. Figure~\ref{fig:double1} shows the case where only one edge of the triangle is an intra-community edge. Figure~\ref{fig:double2} shows that the two inter-community edges must form a double edge in the macroscopic configuration model. The last possibility is that all three edges of the triangle are inter-community edges as in Figure~\ref{fig:triang1}. This corresponds to a triangle in the macroscopic configuration model (Figure~\ref{fig:triang2}).

\begin{figure}[tb]
\centering
\begin{subfigure}[b]{0.10\textwidth}
\centering
\includegraphics[width=\textwidth]{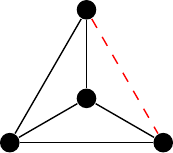}
\captionsetup{width=1.15\textwidth}
\caption{self-loop}
\label{fig:self1}
\end{subfigure}
\hspace{0.8cm}
\begin{subfigure}[b]{0.13\textwidth}
\centering
\includegraphics[width=0.47\textwidth]{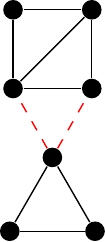}
\captionsetup{width=1.4\textwidth}
\caption{double edge}
\label{fig:double1}
\end{subfigure}
\hspace{0.8cm}
\begin{subfigure}[b]{0.15\textwidth}
\centering
\includegraphics[width=\textwidth]{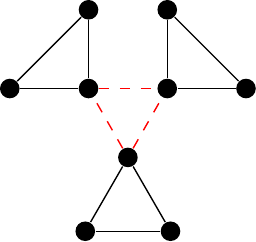}
\caption{triangle}
\label{fig:triang1}
\end{subfigure}
\caption{Possibilities to form triangles in the hierarchical configuration model that are not entirely inside communities. Edges between communities (dashed) that add clustering correspond to either a self-loop, a double edge or a triangle in the macroscopic configuration model.}
\label{fig:clustcm}
\end{figure}
\begin{figure}[tb]
\centering
\begin{subfigure}[b]{0.14\textwidth}
\centering
\includegraphics[width=\textwidth]{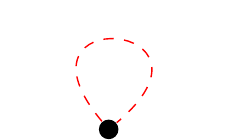}
\caption{self-loop}
\label{fig:self2}
\end{subfigure}
\hspace{0.8cm}
\begin{subfigure}[b]{0.1\textwidth}
\centering
\includegraphics[width=0.4\textwidth]{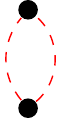}
\captionsetup{width=1.6\textwidth}
\caption{double edge}
\label{fig:double2}
\end{subfigure}
\hspace{1.2cm}
\begin{subfigure}[b]{0.09\textwidth}
\centering
\includegraphics[width=0.7\textwidth]{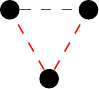}
\captionsetup{width=1.3\textwidth}
\caption{triangle}
\label{fig:triang2}
\end{subfigure}
\caption{Figure~\ref{fig:clustcm} on macroscopic level. The inter-community edges that add clustering correspond to either a self-loop, a double edge or a triangle of the macroscopic configuration model.}
\label{fig:clust2}
\end{figure}

Hence, either the triangle was present in $H$ already, or it corresponds to a double edge, self-loop or triangle in $\phi(G)$. Here we recall that $\phi(G)$ denotes the macroscopic configuration model. Let the number of self-loops, double edges and triangles in $\phi(G)$ be denoted by $W^{\sss{(n)}}, M^{\sss{(n)}}$ and $T^{\sss{(n)}}$ respectively. Denote the number of triangles entirely in communities of $G$ by $T^{\sss{(n)}}_{\text{com}}$. The number of triangles in $G$ is bounded from below by $T_{\text{com}}^{\sss{(n)}}$.
Using~\eqref{eq:clustcom}, we obtain that
	\begin{equation}\label{eq:Tcom}
	3T_{\text{com}}^{\sss{(n)}}=\sum_{H}\sum_{v\in V_H}n^{\sss{(n)}}_{H} P_v^{\sss{(H)}}=\sum_{H}\sum_{v\in V_H}n_{H} C_H d_v^{\sss{(c)}}(d_v^{\sss{(c)}}-1).
	\end{equation}
Since $P_v^{\sss{(H)}}\leq d_v^2$, and $\lim_{N\to\infty}\mathbb{E}[\hat{D}_N]=\mathbb{E}[\hat{D}]$,
\begin{equation}\label{eq:Cnlow}
\begin{aligned}[b]
C^{\sss{(n)}}&=
\frac{3\times\text{number of triangles in }G}{\text{number of connected triples in }G}
\geq \frac{3T^{\sss{(n)}}_{\text{com}}/n}{\sum_{H}n^{\sss{(n)}}_{H}\sum_{v\in V_H}d_v(d_v-1)/(2n)}\\
&\inprob 
\frac{2\sum_{H}\sum_{v\in V_H}P(H)C_H d_v^{\sss{(c)}}(d_v^{\sss{(c)}}-1)}{\sum_{H}\sum_{v\in V_H} P(H)d_v(d_v-1)}.
\end{aligned}
\end{equation}
The sums in~\eqref{eq:Cnlow} are finite due to the assumptions $\mathbb{E}[D^2]<\infty$ and $\mathbb{E}[\hat{D}^2]<\infty$.

For the upper bound, we use that every self-loop on the community level adds at most $s_i-2$ triangles, and every triangle and double edge on the community level adds at most one triangle. This yields the inequality
\begin{equation}\label{eq:ntriang}
\text{number of triangles } G\leq T^{\sss{(n)}}_{\text{com}}+M^{\sss{(n)}}+S^{\sss{(n)}}+\sum_{i=1}^{W^{\sss{(n)}}}(s_{\mathcal{I}_i}-2).
\end{equation}
Here the sum is over all communities where a self-loop is present, written as $(\mathcal{I}_i)_{i=1}^{W^{\sss{(n)}}}$. If a community has multiple self-loops, then the community is counted multiple times in the sum.
By~\cite[Theorem~5]{ball2010}
\begin{equation}\label{eq:cout}
\left(M^{\sss{(n)}}+T^{\sss{(n)}}\right)/n\inprob 0.
\end{equation}
in a configuration model with $\mathbb{E}[D^2]<\infty$. 

The last term in~\eqref{eq:ntriang} satisfies
\begin{align}\label{eq:cself}
\frac{\sum_{i=1}^{W^{\sss{(n)}}}(s_{\mathcal{I}_i}-2)}{n}&=\frac{W^{\sss{(n)}}\mathbb{E}[S_n-2\mid\text{self-loop}]}{n}\leq\frac{W^{\sss{(n)}}\max_{i\in[n]}s_i}{n}=\frac{W^{\sss{(n)}}o(n)}{n}\inprob 0.
\end{align}
The last equality follows because $\mathbb{E}[S_n]\to\mathbb{E}[S]<\infty$, which implies that 
\begin{equation}
\lim_{k\to\infty}\lim_{n\to\infty}\frac{1}{n}\sum_{j\in[n]}s_j\mathbbm{1}\{s_j>k\}=0,
\end{equation}
so that $\max_i s_i=o(n)$. The convergence follows since the number of self-loops in a configuration model converges to a Poisson distribution with mean $\nu_D$~\cite[Proposition~7.11]{hofstad2009}, combined with $\mathbb{E}[D^2]<\infty$.

Combining~\eqref{eq:cout} and~\eqref{eq:cself} yields
\begin{align}\label{eq:clow}
C^{\sss{(n)}}&\leq\frac{3T^{\sss{(n)}}_{\text{com}}+3(M^{\sss{(n)}}+T^{\sss{(n)}})+3\sum_{i=1}^{W^{\sss{(n)}}}(s_{\mathcal{I}_i}-2)}{\sum_{H}n^{\sss{(n)}}_{H}\sum_{v\in V_H}d_v(d_v-1)/2} \nonumber\\
&\inprob 
\frac{2\sum_{H} P(H)\sum_{v\in V_H}C_Hd_v^{\sss{(c)}}(d_v^{\sss{(c)}}-1)}{\sum_{H} P(H)\sum_{v\in V_H}d_v(d_v-1)}.
\end{align}
Together with~\eqref{eq:Cnlow} this proves~\eqref{eq:clust}.

To prove~\eqref{eq:clustk}, a similar argument can be used. The number of connected neighbors of vertices of degree $k$ is bounded from below by $\sum_Hn_H^{\sss{(n)}}\sum_{v\in V_H:d_v=k}P_v^{\sss{(H)}}$, and from above by
\begin{equation}
\sum_Hn_H^{\sss{(n)}}\sum_{v\in V_H:d_v=k}P_v^{\sss{(H)}}+M^{\sss{(n)}}+S^{\sss{(n)}}+\sum_{i=1}^{W^{\sss{(n)}}}(s_{\mathcal{I}_i}-2).
\end{equation}
Then, dividing by $k(k-1)n\hat{p}_k^{\sss{(n)}}$, where $\hat{p}_k^{\sss{(n)}}$ is the probability of having a vertex of degree $k$, and taking the limit yields~\eqref{eq:clustk}. Note that the assumption that $\mathbb{E}[\hat{D}^2]<\infty$ is not necessary for this clustering coefficient, since $P_v/k(k-1)\leq 1$ for all vertices of degree $k$.
\end{proof}

\section{Percolation}\label{sec:perc}

We now consider bond percolation on $G$, where each edge of $G$ is removed independently with probability $1-\pi$. We are interested in the critical percolation value and the size of the largest percolating cluster. Percolation on the configuration model was studied in~\cite{fountoulakis2007, janson2008}. Here we extend these results to the hierarchical configuration model.

Percolating $G$ is the same as first percolating only the edges within communities, and then percolating the edges between communities.
For percolation inside a community, only the edges inside a community are removed with probability $1-\pi$. The half-edges attached to a community are not percolated. Let $H_\pi$ denote the subgraph of $H$, where each edge of $H$ has been deleted with probability $1-\pi$. When percolating a community, it may split into different connected components. Let $g(H,v,l,\pi)$ denote the probability that the component of $H_\pi$ containing $v$ has inter-community degree $l$. If $H_{\pi}$ is still connected, then the component containing $v$ still has $d_H$ outgoing edges for all $v\in V_H$. If $H_\pi$ is disconnected, then this does not hold. If one of the components of $H_\pi$ has an outgoing edge, each vertex in another component of $H_\pi$ cannot reach that edge. Therefore, a vertex in this other component is connected to less than $d_H$ outgoing edges. 

To compute the size of the largest percolating cluster, we need the following definitions:
\begin{align}
p_k' &:=\frac{\sum_{H}\sum_{v\in V_H}d_v^{\sss{(b)}}P(H)g(H,v,k,\pi)/k}{\sum_{H}\sum_{v\in V_H}\sum_l d_v^{\sss{(b)}}P(H)g(H,v,l,\pi)/l},\label{eq:pk}\\
h(z)&:=\sum_{k=1}^\infty k p'_kz^{k-1},\\
\lambda&:=\sum_{k=0}^\infty kp_k'.
\end{align}
The probabilities $(p'_k)_{k\geq 0}$ can be interpreted as the asymptotic probability distribution of the inter-community degrees of the connected parts of communities after percolation inside communities. Then $h(z)$ and $\lambda$ are the derivative of the probability generating function and the mean of the inter-community degrees of the components of communities after percolation respectively.

Define $D_\pi^*$ as the number of inter-community edges after entering a percolated community from a randomly chosen edge. The probability of entering at vertex $v$ in community $H$, equals $P(H){d_v^{\sss{(b)}}}/{\mathbb{E}[D]}$. After entering $H$ at vertex $v$, there are in expectation $\sum_{k=1}^{D_H-1}kg(H,v,k+1,\pi)$ edges to other communities (since one edge was used to enter $H$). Hence, 
\begin{equation}\label{eq:pic1}
	\mathbb{E}[D_\pi^*]=\frac{1}{\mathbb{E}[D]}\sum_{H}P(H)\sum_{v\in V_H}d_v^{\sss{(b)}}\sum_{k=1}^{D_H-1}kg(H,v,k+1,\pi).
\end{equation}
After percolating the inter-community edges, a fraction of $\pi$ of these edges remain. Thus, after percolating all edges, when entering a community, the expected number of outgoing edges excluding the traversed edge  is $\pi\mathbb{E}[D_\pi^*]$. We expect the critical value of $\pi$ to satisfy $\pi\mathbb{E}[D_\pi^*]=1$, i.e., the expected number of edges to other communities is one, after entering a community from a randomly chosen edge. The next theorem states that this is indeed the critical percolation value:

\begin{theorem}\label{thm:perc}
Assume $G$ is a hierarchical configuration model satisfying Conditions\textup{~\ref{cond:graph}} and\textup{~\ref{cond:size}}. The critical value of the percolation parameter $\pi_c$ of $G$ satisfies
 \begin{equation}\label{eq:pic}
 \pi_c=\frac{1}{\mathbb{E}[D^*_{\pi_c}]}.
 \end{equation}
Furthermore:
\begin{enumerate}
\item
For $\pi>\pi_c$, the size of the largest component of the percolated graph satisfies
\begin{equation}\label{eq:percsize}
\frac{v(\mathscr{C}_{\textup{max}})}{N}\inprob\frac{1}{\mathbb{E}[S]}\sum_{k=1}^\infty\sum_{H}\sum_{v\in V_H} P(H)g(H,v,k,\pi)\left(1-(1-\sqrt{\pi}+\sqrt{\pi}\xi)^k\right)>0,
\end{equation}
where $\xi$ is the unique solution in $(0,1)$ of 
\begin{equation}\label{eq:xi}
\sqrt{\pi}h(1-\sqrt{\pi}+\sqrt{\pi}\xi)+(1-\sqrt{\pi})\lambda=\lambda\xi.
\end{equation}
\item
For $\pi\leq\pi_c$, $v(\mathscr{C}_{\textup{max}})/N\inprob 0$. 
\end{enumerate}
\end{theorem}
Note that for the standard configuration model,~\eqref{eq:pic} simplifies to $\pi_c=\mathbb{E}[D]/\mathbb{E}[D(D-1)]$, since in that case, for a vertex of degree $d$, $g(v,v,k,\pi)=\mathbbm{1}_{\{k=d\}}$. 
Furthermore, $\pi_c=0$ when for any $\pi>0$, the expected number of edges to other communities is infinite when entering a community via a uniformly chosen edge. 

\begin{proof}
The proof of Theorem~\ref{thm:perc} has a similar structure as the proof of~\cite[Theorem 1]{coupechoux2014}. 
The proof consists of three key steps:
\begin{enumerate}[label=\normalfont({\alph*})]
	\item
	First, each edge within each community is removed with probability $1-\pi$. This may split the community into several connected components. We find the distribution of the inter-community degrees of the connected components of the percolated communities, which is given by $p'_k$, as in~\eqref{eq:pk}. We identify vertices that are in the same connected component of a community. Lemma~\ref{lem:pk} shows that this results in a graph $\phi(G_\pi)$ that is distributed as a configuration model with asymptotic degree probabilities $p_k'$ (recall~\eqref{eq:pk}).
	\item
	We then remove each edge between communities with probability $1-\pi$. Results of~\cite{janson2008} can now be applied to the configuration model with distribution $p_k'$ to find the critical percolation value and the size of the giant hierarchical component.
	\item
	Next, we translate the number of communities in the largest percolated hierarchical component to the number of vertices. Then we show that this is indeed the largest component of the percolated graph.	
\end{enumerate}

\paragraph*{Auxiliary graph.}

We introduce the auxiliary graph $\phi(\bar{G})$, defined for every subgraph $\bar{G}\subset G$, and obtained by identifying the vertices that belonged to the same community in $G$, and are connected in $\bar{G}$~\cite{coupechoux2014}. Hence, in $\phi(\bar{G})$ every vertex represents a connected part of a community. Figure~\ref{fig:phi} illustrates $\phi(\bar{G})$. For a hierarchical configuration model $G$, the graph $\phi(G)$ is a configuration model where communities of $G$ are collapsed into single vertices. 
\begin{figure}[tb]
	\centering
	\includegraphics[width=0.4\textwidth]{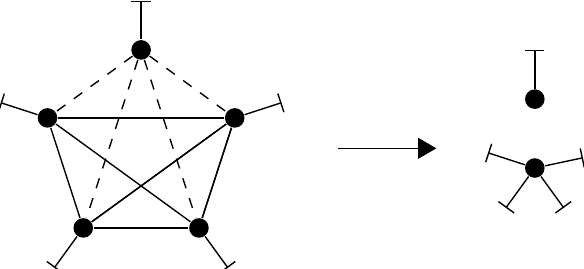}
	\caption{Left, a subgraph $\bar{H}$  of a community $H$, dashed lines are not in $\bar{H}$, but are present in $H$, solid lines are present in $\bar{H}$. The graph 				$\phi(\bar{H})$ is shown on the right.}
	\label{fig:phi}
\end{figure}

\begin{lemma}\label{lem:pk}
	Let $G$ be a hierarchical configuration model satisfying Conditions\textup{~\ref{cond:graph}} and\textup{~\ref{cond:size}}. Let $G_\pi$ denote the subgraph of $G$ where each edge inside each community is removed with probability $1-\pi$. Then the graph $\phi(G_\pi)$ is distributed as a configuration model with degree probabilities $p'_k$ given in~\eqref{eq:pk}.
\end{lemma}

\begin{proof}
	We independently delete each edge within each community with probability $1-\pi$. We want to find the degree distribution of $\phi(G_{\pi})$. Let $M^{\sss{(n)}}(H,v,k,\pi)$ denote the number of connected components of the percolated versions of community $H$ containing vertex $v$ and having inter-community degree $k$. Each community of shape $H$ has an equal probability that the component containing $v$ has inter-community degree $k$ given by $g(H,v,k,\pi)$. Furthermore, the probability that a randomly chosen community has shape $H$ is independent of the probability that the inter-community degree is $k$ after percolation in a community of shape $H$. Therefore, given the number of type $H$ communities $n^{\sss{(n)}}_H$, $M^{\sss{(n)}}(H,v,k,\pi)\sim \text{Bin}(n^{\sss{(n)}}_H,g(H,v,k,\pi))$. Thus, by the weak law of large numbers,
	\begin{equation}\label{eq:mn} 
	\frac{M^{\sss{(n)}}(H,v,k,\pi)}{n}=\frac{M^{\sss{(n)}}(H,v,k,\pi)}{n^{\sss{(n)}}_H}\frac{n^{\sss{(n)}}_H}{n}\inprob P(H)g(H,v,k,\pi).
	\end{equation}
	Let $N^{\sss{(n)}}(H,k,\pi)$ denote the total number of connected components of the percolated versions of $H$ having inter-community degree $k$. This number can be obtained by counting the number of half-edges of all connected components of percolated graphs with inter-community degree $k$, and then dividing by $k$. Each vertex $v$ in such a percolated community contributes $d_v^{\sss{(b)}}$ to the inter-community degree of the percolated community. Thus,
	\begin{equation}\label{eq:nn}
	N^{\sss{(n)}}(H,k,\pi)=\sum_{v\in{V_H}}d_v^{\sss{(b)}}M^{\sss{(n)}}(H,v,k,\pi)/k.
	\end{equation}
	Let $\tilde{n}$ denote the number of vertices in $\phi(G_{\pi})$, so that $\tilde{n}=\sum_{H}\sum_k N^{\sss{(n)}}(H,k,\pi)$. Similarly, the number vertices of degree $k$ in $\phi(G_\pi)$ is denoted by $\tilde{n}_k=\sum_{H}N^{\sss{(n)}}(H,k,\pi)$. 
	Furthermore, $\sum_k N^{\sss{(n)}}(H,k,\pi)/n\leq P_n(H)s_H$, and therefore by Condition~\ref{cond:graph},~\eqref{eq:mn} and~\eqref{eq:nn},
\begin{equation}\label{eq:ntilde}
\tilde{n}/n\inprob \sum_{H}\sum_{k=1}^{D_H}\sum_{v\in V_H} d_v^{\sss{(b)}}P(H)g(H,v,k,\pi)/k.
\end{equation}
Therefore also
	\begin{equation}
	\frac{N^{\sss{(n)}}(H,k,\pi)}{\tilde{n}}=\frac{N^{\sss{(n)}}(H,k,\pi)/n}{\tilde{n}/n}\inprob\frac{\sum_{v\in V_H}d_v^{\sss{(b)}}P(H)g(H,v,k,\pi)/k}{\sum_{H}\sum_{v\in V_H}\sum_l d_v^{\sss{(b)}}P(H)g(H,v,l,\pi)/l}.
	\end{equation}
	Hence, the proportion of vertices in $\phi(G_\pi)$ with degree $k$ tends to
	\begin{equation}\label{eq:pd}
	\frac{\tilde{n}_k}{\tilde{n}} =\sum_{H} \frac{N^{\sss{(n)}}(H,k,\pi)}{\tilde{n}}\inprob
	p_k' .
	\end{equation}
	Since the edges between communities in $G$ were paired at random, this means that the graph $\phi(G_\pi)$ is distributed as a configuration model with degree probabilities $p_k'$.
\end{proof}

Using Lemma~\ref{lem:pk}, we now prove Theorem~\ref{thm:perc}:\\
	\leavevmode\\
\textbf{Step (a).}
Lemma~\ref{lem:pk} proves that $\phi(G_\pi)$ is distributed as a configuration model with degree probabilities $p'_k$. 
\\\\
\textbf{Step (b).}
$\phi(G_\pi)$ and $\phi(G)$ have $\sum_kk\tilde{n}_k$ and $\sum_k kn_k$ half-edges, respectively. Since only edges inside communities have been deleted, $\sum_k kn_k$ equals $\sum_kk\tilde{n}_k$. By Condition~\ref{cond:size}\ref{enum:ED}, $\sum_k kn_k/n\to\mathbb{E}[D]$. Furthermore, by~\eqref{eq:ntilde} $\tilde{n}/n$ converges, 
hence $\sum_k k\tilde{n}_k/\tilde{n}$ converges. Therefore we can apply Theorem~3.9 from~\cite{janson2008}, which states that after percolation, a configuration model with degree probabilities $p_k'$ has a giant component if
\begin{equation}
	\pi\sum_k k(k-1)p_k'>\sum_k kp_k'.
\end{equation}
From Theorem~\ref{thm:size} we know that a giant hierarchical component is also a giant component in $G$, and a hierarchical component of size $o_{\mathbb{P}}(n)$ is a component of size $o_{\mathbb{P}}(N)$. Hence, the giant component emerges precisely when the giant hierarchical component emerges. 
Substituting~\eqref{eq:pd} gives for the critical percolation value $\pi_c$ that,
\begin{equation}
\begin{aligned}[b]
\pi_c&=\frac{\sum_k kp_k'}{\sum_k k(k-1)p_k'}=
\frac{\sum_{H}\sum_{v\in V_H}d_v^{\sss{(b)}}\sum_{k=1}^{D_H} P(H)g(H,v,k,\pi_c)k/k}{\sum_{H}\sum_{v\in V_H}d_v^{\sss{(b)}}\sum_{k=1}^{D_H} P(H)g(H,v,k,\pi_c)k(k-1)/k}\\
&=
\frac{\sum_{H}\sum_{v\in V_H}d_v^{\sss{(b)}}P(H)}{\sum_{H}\sum_{v\in V_H}d_v^{\sss{(b)}}\sum_{k=1}^{D_H} P(H)g(H,v,k,\pi_c)(k-1)}\\
&=
\frac{\mathbb{E}[D]}{\sum_{H}\sum_{v\in V_H}\sum_{k=1}^{D_H-1} d_v^{\sss{(b)}}P(H)g(H,v,k+1,\pi_c)k}\\
&=\frac{1}{\mathbb{E}[D_{\pi_c}^*]}.
\end{aligned}
\end{equation}
\textbf{Step (c).}
Now assume that $\pi>\pi_c$. The number of degree $r$ vertices in the largest component of $\phi(G_\pi)$ satisfies $v_r(\mathscr{C}^{\sss{\text{H}}}_{\text{max}})/\tilde{n}\inprob\sum_{l\geq r}b_{lr}(\sqrt{\pi})p_l'(1-\xi^r)$~\cite{janson2008}, with $b_{lr}(\sqrt{\pi})={l \choose r}\sqrt{\pi}^r(1-\sqrt{\pi})^{l-r}$ is the probability that a binomial with parameters $l$ and $\sqrt{\pi}$ takes value $r$, and $\xi$ is as in~\eqref{eq:xi}. To translate the number of percolated communities in the largest component into the number of vertices in the largest component, we want to know the expected number of vertices in the largest component that are in a percolated community with inter-community degree $k$. The size of a percolated community is independent of being in the largest hierarchical component, but does depend on the inter-community degree of the percolated community. The total number of vertices in connected percolated components with inter-community degree $k$ is given by $\sum_{H}\sum_{v\in V_H}M^{\sss{(n)}}(H,v,k,\pi)$, and the total number of percolated communities of inter-community degree $k$ is given by $\sum_{H}N^{\sss{(n)}}(H,k,\pi)$. Furthermore, $\sum_{v\in V_H}M^{\sss{(n)}}(H,v,k,\pi)/n\leq P_n(H)s_H$. Hence, by Condition~\ref{cond:size}, the expected size of a percolated community, given that it has inter-community degree $k$, satisfies
\begin{equation}
\begin{aligned}[b]
	\mathbb{E}[S_\pi\mid\text{inter-community degree } k]
	&=\frac{\sum_{H}\sum_{v\in V_H}M^{\sss{(n)}}(H,v,k,\pi)}{\sum_{H}N^{\sss{(n)}}(H,k,\pi)}\\
	&=\frac{\sum_{H}\sum_{v\in V_H}M^{\sss{(n)}}(H,v,k,\pi)/n}{\sum_{H}N^{\sss{(n)}}(H,k,\pi)/n}\\
	&\inprob\frac{\sum_{H}\sum_{v\in V_H} P(H)g(H,v,k,\pi)}{\sum_{H}\sum_{v\in V_H}d_v^{\sss{(b)}}P(H)g(H,v,k,\pi)/k}.
	\end{aligned}
\end{equation}
Since $v_r(\mathscr{C}_{\textup{max}})/N\leq \hat{p}_r^{\sss{(n)}}$, which sums to one, we can compute the asymptotic number of vertices in the largest component of $\phi(G_\pi)$ as
\begin{align}
\frac{v(\mathscr{C}_{\textup{max}})}{N}&=\sum_{r=0}^\infty\frac{v_r(\mathscr{C}_{\textup{max}})/\tilde{n}}{N/n\cdot n/\tilde{n}}\inprob
\sum_{r=0}^\infty\frac{\sum_{l\geq r}b_{lr}(\sqrt{\pi})p_l'(1-\xi^r)\mathbb{E}[S_\pi\mid\text{inter-community degree } l]}{\mathbb{E}[S]/\sum_{H}\sum_{v\in V_H}\sum_k d_v^{\sss{(b)}}P(H)g(H,v,k,\pi)/k}\nonumber\\
&=
\frac{\sum_{r=0}^\infty\sum_{l\geq r}b_{lr}(\sqrt{\pi})(1-\xi^r)p_l'}{\mathbb{E}[S]/\sum_{H}\sum_{v\in V_H}\sum_k d_v^{\sss{(b)}}P(H)g(H,v,k,\pi)/k}\frac{\sum_{H}\sum_{v\in V_H} P(H)g(H,v,l,\pi)}{\sum_{H}\sum_{v\in V_H}d_v^{\sss{(b)}}P(H)g(H,v,l,\pi)/l}\nonumber\\
&=
\frac{\sum_{l=0}^\infty\sum_{r=0}^lb_{lr}(\sqrt{\pi})(1-\xi^r)\frac{p'_l}{p'_l}\sum_{H}\sum_{v\in V_H}P(H)g(H,v,l,\pi)}{\mathbb{E}[S]}\nonumber\\
&=
\frac{\sum_{l=0}^\infty (1-(1-\sqrt{\pi}+\sqrt{\pi}\xi)^l)\sum_{H}\sum_{v\in V_H}P(H)g(H,v,l,\pi)}{\mathbb{E}[S]}.
\end{align}
Any other component of $\phi(G_\pi)$ has size $o_\Prob(\tilde{n})$ by~\cite[Theorem~3.9]{janson2008}. 
As shown in the proof of Theorem~\ref{thm:size}, any component of size $o_{\Prob}(\tilde{n})$ in $\phi(G_\pi)$ is a component of size $o_{\Prob}(N)$ in the total graph. Hence, w.h.p. $\mathscr{C}_{\textup{max}}$ is the largest component of the percolated graph $G_\pi$. 

When $\pi<\pi_c$, the largest component of $\phi(G_\pi)$ satisfies $v(\mathscr{C}^{\sss{\text{H}}}_{\text{max}})/\tilde{n}\inprob 0$~\cite{janson2008}. Again, by the analysis of Theorem~\ref{thm:size}, this component is of size $o_\Prob(N)$ in the original graph.
\end{proof}

Equation~\eqref{eq:xi} also has an intuitive explanation.
Let $Q$ be the distribution of the community inter-community degrees after percolation when following a randomly chosen half-edge. Then we can interpret $\xi$ as the extinction probability of a branching process with offspring distribution $Q$. Percolating the inter-community edges with probability $1-\pi$ is the same as deleting each half-edge with probability $1-\sqrt{\pi}$. Then, with probability $1-\sqrt{\pi}$ the randomly chosen half-edge is paired to a deleted half-edge, in which case the branching process goes extinct. With probability $\sqrt{\pi}$, the half-edge leads to a half-edge which still exists after percolation, and leads to a community. The probability generating function of the number of half-edges pointing out of this community before percolating the half-edges is $\frac1\lambda h(\xi)$. Since the number of half-edges after percolation is binomial given the number of half-edges that were present before percolation, the probability generating function of the number of half-edges pointing out of a community entered by a randomly chosen half-edge is $\frac1\lambda h(1-\sqrt{\pi}+\sqrt{\pi}\xi)$. Combining this yields~\eqref{eq:xi}. 

\paragraph{The case $\mathbb{E}[D^2]=\infty$.}\label{sec:infD}

In the standard configuration model $\pi_c=0$ precisely when $\mathbb{E}[D^2]=\infty$. In the hierarchical configuration model, this may not be true, since it is possible to construct communities with large inter-community degrees, while all individual vertices have a low degree. An example of such a community structure is the hierarchical configuration model, where each community is a line graph $H_L$ of $L$ vertices with probability $\bar{p}_L$, where each vertex has inter-community degree one. Figure~\ref{fig:ED2} illustrates $H_L$ for $L=5$. We assume that $\bar{p}_L$ obeys the power law $\bar{p}_L=cL^{-\alpha}$, with $\alpha\in(2,3)$. Then $\mathbb{E}[D]<\infty$, but $\mathbb{E}[D^2]=\infty$. Hence, communities may have large inter-community degrees. However, $G$ is a 3-regular graph, so no individual vertex has high degree. From this fact, we can already conclude that $\pi_c\neq0$. Suppose $\pi_c<\frac12$. Then, after percolation every vertex has less than two expected neighbors. Hence, there is no giant component w.h.p.
We can also use Theorem~\ref{thm:perc} to show that $\pi_c\neq0$. We compute the denominator of~\eqref{eq:pic}, and show that it is finite. We have
\begin{equation}
	\sum_{v\in V_H}d_v^{\sss{(b)}}g(H_L,v,k,\pi_c)=\begin{cases}
	2k\pi_c^{k-1}(1-\pi_c)+k\pi_c^{k-1}(1-\pi_c)^2(L-k-1) &\text{if } k<L,\\
	k\pi_c^{k-1}&\text{if } k=L.
	\end{cases}
\end{equation}
This gives
\begin{equation}
	\sum_{k=1}^{L-1}k\sum_{v\in V_H}d_v^{\sss{(b)}}g(H_L,v,k+1,\pi_c)=\frac{2\pi_c(\pi_c^L+L(1-\pi_c))-1}{(1-\pi_c)^2}.
\end{equation}
Using that $\bar{p}_l=cl^{-\alpha}$ gives for~\eqref{eq:pic1}
\begin{align}
	\mathbb{E}[D_{\pi_c}^*]&=\frac{1}{\mathbb{E}[D]}\sum_{H}P(H)\sum_{v\in V_H}d_v^{\sss{(b)}}\sum_{k=1}^{D_H-1}kg(H,v,k+1,\pi_c)\nonumber\\
	&=\frac{1}{\mathbb{E}[D]}\sum_{L=1}^\infty cL^{-\alpha} \frac{2\pi_c(\pi_c^L+L(1-\pi_c))-1}{(1-\pi_c)^2}\nonumber\\
	&=\frac{1}{\mathbb{E}[D]}\frac{2\pi_c}{(1-\pi_c)^2}\left(-1+(1-\pi_c)\mathbb{E}[D]+\sum_{L=1}^\infty c\pi_c^LL^{-\alpha}\right)\label{eq:edinf}.
\end{align}
From~\eqref{eq:edinf} we see that $\pi_c=0$ is not a solution of~\eqref{eq:pic}. Hence, $\pi_c\neq 0$, even though $\mathbb{E}[D^2]=\infty$.

\begin{figure}[tb]
\centering
\begin{subfigure}[t]{0.3\linewidth}
\centering
\includegraphics[width=0.8\textwidth]{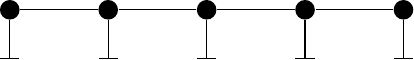}
\caption{ }
\label{fig:ED2}
\end{subfigure}
\hspace{0.2cm}
\begin{subfigure}[t]{0.3\linewidth}
\centering
\includegraphics[width=0.40\textwidth]{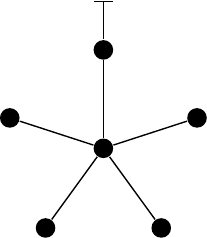}
\caption{ }
\label{fig:infp}
\end{subfigure}
\hspace{0.2cm}
\begin{subfigure}[t]{0.3\linewidth}
\centering
\includegraphics[width=0.48\textwidth]{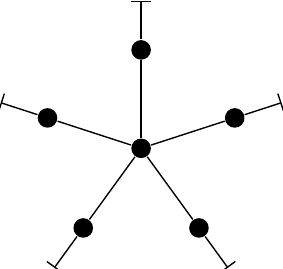}
\caption{ }
\label{fig:star}
\end{subfigure}
\caption{Communities with $L=5$.}
\end{figure}

\paragraph{Infinite second moment of degree.}\label{sec:infp}

When the second moment of the degree distribution as defined in Proposition~\ref{prop:deg} is infinite, $\pi_c$ also does not have to be zero. It is possible to `hide' all vertices of high degree inside communities that have small inter-community degrees. The small inter-community degrees make it difficult to leave the community in percolation. One example of such a community structure is the case in which each community is a star-shaped graph with $L$ endpoints with probability $p_L$. One vertex in the graph has inter-community degree one, and all the other vertices have inter-community degree zero. Figure~\ref{fig:infp} illustrates the star-shaped graph for $L=5$. Since each community has only one outgoing edge, there cannot be a giant component in $G$. We can also see this from Theorem~\ref{thm:size}, since $\mathbb{E}[D]=\mathbb{E}[D^2]=1$. By Proposition~\ref{prop:deg}, the degree distribution equals
\begin{equation}
\hat{p}_k=\begin{cases} \left(\sum_{L=1}^\infty (L-1)p_L+p_1\right)/\mathbb{E}[S] &\text{if } k=1,\\
 \left(1+p_2\right)\mathbb{E}[S] &\text{if } k=2,\\
 p_k/\mathbb{E}[S] &\text{if } k>2.
\end{cases}
\end{equation}
When $p_l$ is a probability distribution with infinite second moment, the second moment of $\hat{p}_l$ is also infinite. Hence, the degree distribution of the hierarchical configuration model $G$ has infinite second moment, while there is no giant component, so that certainly $\pi_c\neq 0$.

\paragraph{A sufficient condition for $\pi_c=0$.}
\mbox{}
\newline
By~\eqref{eq:pic},
\begin{equation}
\begin{aligned}[b]
\pi_c&=\frac{\mathbb{E}[D]}{\sum_{G}P(H)\sum_{v\in V_H}d_v^{\sss{(b)}}\sum_{k=1}^{D_H-1}kg(H,v,k+1,\pi_c)}\\
&\leq\frac{\mathbb{E}[D]}{\sum_{G}P(H)\sum_{v\in V_H}{(d_v^{\sss{(b)}}})^2}
\end{aligned}
\end{equation}
Hence, $\sum_{H}\sum_{v\in V_H}P(H){(d_v^{\sss{(b)}}})^2=\infty$ is a sufficient condition for $\pi_c=0$. This condition can be interpreted as an infinite second moment of the inter-community degrees of individual vertices. However, it is not a necessary condition. It is possible to construct a community where all individual vertices have a small inter-community degree, but are connected to a vertex with high degree. Consider for example the star community of Figure~\ref{fig:star}, with one vertex in the middle, linked to $L$ other vertices. The $L$ other vertices have inter-community degree one, and the middle vertex has inter-community degree zero, hence all vertices have a small inter-community degree. However, the middle vertex can have a high degree. Let each community be a star-shaped community with $L$ outgoing edges with probability $\bar{p}_L$. We can calculate that $\pi_c=\sum_L\bar{p}_LL(L-1)\pi_c^2$. Hence, if we choose $\bar{p}_L$ with finite first moment and infinite second moment, $\pi_c=0$. However, $\sum_{H}\sum_{v\in V_H}{(d_v^{\sss{(b)}})}^2=\sum_LL \bar{p}_L =\mathbb{E}[D]<\infty$. 

\section{Existing graph models with a community structure}\label{sec:exex}

In this section, we show how three existing random graph models with community structure fit within the hierarchical configuration model. 

\subsection{Trapman's household model}\label{sec:house1}

Trapman~\cite{trapman2007} replaces vertices in a configuration model by households in the form of complete graphs, such that the degree distribution of the resulting graph is $p_k$. To achieve this, each community is a single vertex of degree $k$ with probability $(1-\gamma)p_k$, or a complete graph of size $k$ with probability $\gamma\bar{p}_k$. Here $\bar{p}_k$, the probability that a certain clique has degree $k$, is given by 
\begin{equation}\label{eq:pkhouse}
\bar{p}_k=k^{-1}p_k\mathbb{E}[W^{-1}]^{-1},
\end{equation}
where $W$ is a random variable satisfying $\Prob(W=k)=p_k$. Each vertex of the complete graph has one edge to another community. Figure~\ref{fig:house} illustrates a household of size 5. This model is a special case of the hierarchical configuration model with
\begin{equation}\label{eq:hierarchical configuration modeltrap}
H_i=\begin{cases}
(K_{k},(1,\dots,1)) &\text{w.p. }\gamma \bar{p}_k,\\
(v,(k)) &\text{w.p. }(1-\gamma)p_k,
\end{cases}
\end{equation}
where $K_k$ is a complete graph on $k$ vertices. 

We now check when~\eqref{eq:hierarchical configuration modeltrap} satisfies Conditions~\ref{cond:graph} and~\ref{cond:size}.
The assumption $\Prob(D=2)<1$ is satisfied if and only if $p_2<1$. The expected inter-community degree of a community is given by
\begin{equation}
\mathbb{E}[D]=(1-\gamma)\sum_k kp_k+\gamma\sum_k k\bar{p}_k=(1-\gamma)\mathbb{E}[W]+\frac{\gamma}{\mathbb{E}[W^{-1}]}.
\end{equation}
Hence, $\mathbb{E}[D]<\infty$ if $\mathbb{E}[W]<\infty$ and $\mathbb{E}[W^{-1}]\neq 0$. By Jensen's inequality, $\mathbb{E}[W^{-1}]\geq\mathbb{E}[W]^{-1}>0$, hence $\mathbb{E}[D]<\infty$ if and only if $\mathbb{E}[W]<\infty$.
For every community in this model, its size is smaller than or equal to its inter-community degree, so that also $\mathbb{E}[S]<\infty$ if $\mathbb{E}[D]<\infty$. 
Thus, Conditions~\ref{cond:graph} and~\ref{cond:size} hold if $\mathbb{E}[W]<\infty$ and $p_2<1$. Under these conditions we can apply the results for the hierarchical configuration model as derived in Sections~\ref{sec:gen} and~\ref{sec:perc}.

Suppose that $p_k$ follows a power law with exponent $\alpha$. Then $\bar{p}_k$ follows a power law with exponent $\alpha-1$, and the distribution of the inter-community degrees $D$ is a mixture of a power law with exponent $\alpha$ and a power law with exponent $\alpha-1$ by~\eqref{eq:pkhouse}. Thus, the power-law shift of Corollary~\ref{cor:pl} does not occur, since in this household model, the single vertex communities do not satisfy $d_v^{\sss{(b)}}\leq K$. 
For Trapman's household model, the power-law shift only occurs if $\gamma=1$, in which case all communities are households.

\begin{figure}[tb]
	\centering
	\begin{minipage}[t]{0.45\linewidth}
		\centering
		\includegraphics[width=0.4\textwidth]{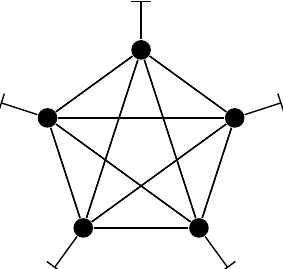}
		\caption{A household of size 5}
		\label{fig:house}
	\end{minipage}
\end{figure}

\subsection{Lelarge and Coupechoux' household model}\label{sec:house2}

Another model that takes complete graphs as communities is the model of Coupechoux and Lelarge~\cite{coupechoux2014}. This model is very similar to Trapman's model. Again, each community is either a complete graph or a single vertex. In contrast to~\cite{trapman2007}, the probability that a certain community is a clique is dependent on the degree of the clique. 
Each vertex of degree $k$ in the macroscopic configuration model is replaced by a complete graph with probability $\gamma_k$. This graph can be modeled as a hierarchical configuration model with
\begin{equation}
H_i=\begin{cases}
(K_{k},(1,\dots,1)) &\text{w.p. }\gamma_k \bar{p}_k\\
(v,(k)) &\text{w.p. }(1-\gamma_k)\bar{p}_k,
\end{cases}
\end{equation}
where $(\bar{p}_k)_{k\geq 1}$ is a probability distribution.
Since the inter-community degrees of all communities have distribution $\Prob(D=k)=\bar{p}_k$, Condition~\ref{cond:size} holds if the probability distribution $\bar{p}_k$ has finite mean and $\bar{p}_2<1$. The size of a community is always smaller than or equal to its inter-community degree, so that also $\mathbb{E}[S]<\infty$ if $\bar{p}_k$ has finite mean. Thus, Conditions~\ref{cond:graph} and~\ref{cond:size} hold if $\bar{p}_k$ has finite mean and $\bar{p}_2<1$.

If these conditions on $\bar{p}_k$ hold, then the degree distribution ${p}_k$ of the resulting graph can be obtained from Proposition~\ref{prop:deg} as 
\begin{equation}
	{p}_k=\frac{(k\gamma_k+(1-\gamma_k))\bar{p}_k}{\sum_{i\geq 0}(i\gamma_i+(1-\gamma_i))\bar{p}_i}.
\end{equation}
Suppose that $\gamma_k\geq\gamma>0$. Then, in contrast to Trapman's household model in Section~\ref{sec:house1}, the degree distribution of the edges between communities, $\bar{p}_k$, follows a power law with exponent $\alpha+1$ if the degree distribution ${p}_k$ follows a power law with exponent $\alpha$.

As an example of such a household model, consider a graph with ${p}_3=a$ and ${p}_6=1-a$ and a tunable clustering coefficient. We take $\gamma_6=0$, but increase $\gamma_3$, while the degree distribution remains the same. Thus, the graph consists of only single vertices of degree 6, single vertices of degree 3 and triangle communities. Since we increase $\gamma_3$, the number of triangles increases, so that also the clustering coefficient increases. 
Figures~\ref{fig:cla75} and~\ref{fig:cla95} show the size of the giant component under percolation for different values of the clustering coefficient using $a=0.75$ and $a=0.95$ respectively. In the case where $a=0.75$, clustering decreases the value of $\pi_c$, whereas if $a=0.95$, clustering increases the value of $\pi_c$. This illustrates that the influence of clustering on bond percolation of a random graph is non-trivial. In two similar random graph models, introducing clustering has a different effect.

\begin{figure}[tb]
	\centering
	\begin{subfigure}[b]{0.45\textwidth}
		\centering
		\includegraphics[width=\textwidth]{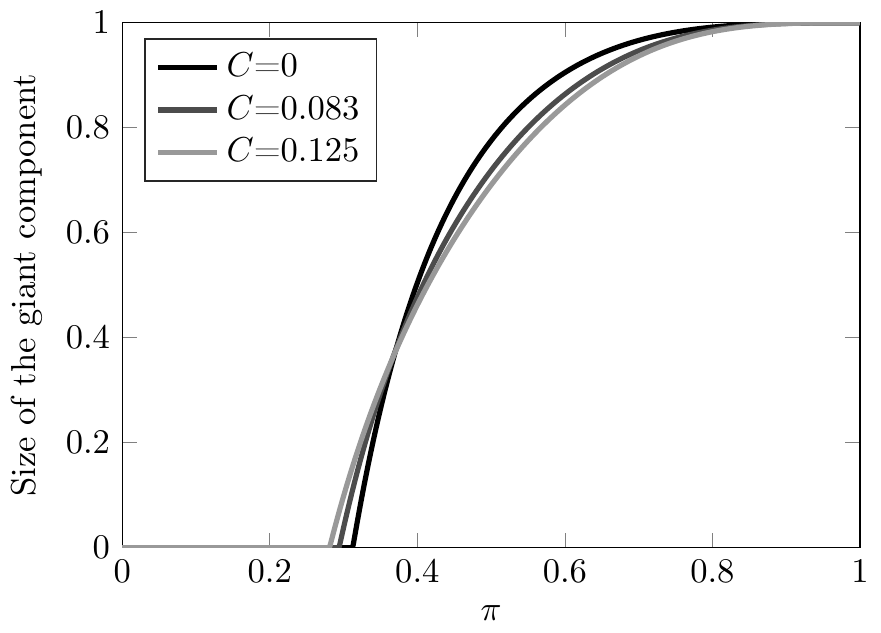}
		\caption{ $a=0.75$}
		\label{fig:cla75}
	\end{subfigure}
	\hspace{0.2cm}
	\begin{subfigure}[b]{0.45\textwidth}
		\centering
		\includegraphics[width=\textwidth]{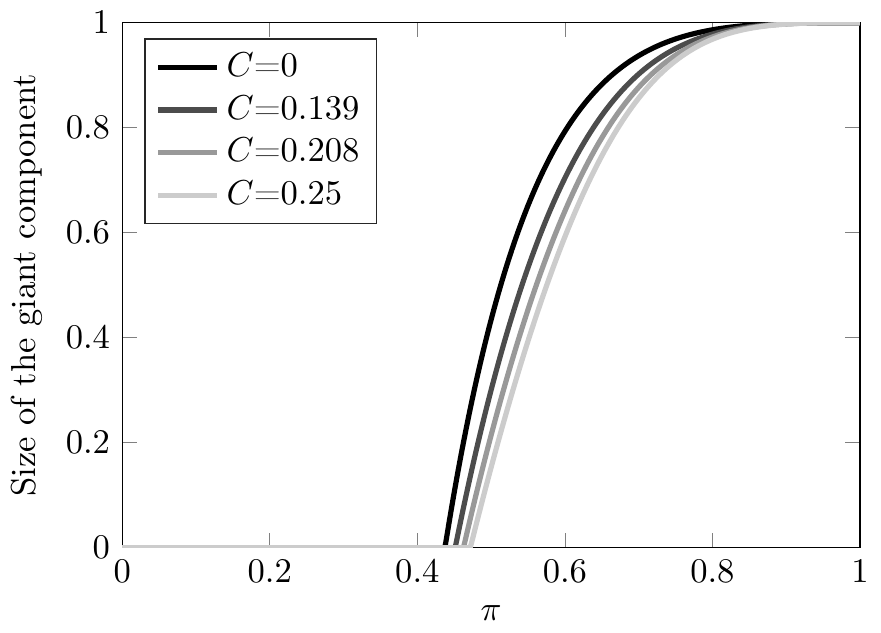}
		\caption{$a=0.95$}
		\label{fig:cla95}
	\end{subfigure}
	\caption{The size of the giant component after bond percolation with probability $\pi$ in a household model with $p_3=a$ and $p_6=1-a$ for various clustering coefficients $C$. If $a=0.75$, adding clustering decreases the critical percolation value, whereas if $a=0.95$, adding clustering increases the critical percolation value.}
\end{figure}

\subsection{Configuration model with triangles}\label{sec:newman}

A third random graph model with clustering is the model by Newman~\cite{newman2009}. In this model, each vertex $v$ has an edge-degree $d_v^{\sss{(1)}}$ and a triangle degree $d_v^{\sss{(2)}}$, denoting the number of triangles that the vertex is part of. Then a random graph is formed by pairing edges at random and pairing triangles at random. 
Even though this model does not explicitly replace vertices in a configuration model by communities, it is also a special case of the hierarchical configuration model if some conditions on the degrees are satisfied. The communities in this model are the connected components consisting only of triangles. Figure~\ref{fig:newmexample} shows two possible realizations of such communities.

From results derived in~\cite{newman2009}, we can find the probability generating function $h_r(z)$ of the number of vertices in triangles that can be reached from a uniformly chosen triangle, and the probability generating function $h_{S^*}(z)$ of the size of the triangle component of a randomly chosen vertex, that together satisfy
\begin{align}
h_r(z)=zg_q(h_r^2(z)),\quad h_{S^*}(z)=zg_p(h_r^2(z)),
\end{align}
where $g_q$ is the probability generating function of the size-biased distribution of the triangle degrees, and $g_p$ the probability generating function of the triangle degree distribution. In the hierarchical configuration model, $h_{S^*}(z)$ can be interpreted as the probability generating function of the size-biased community sizes. Thus, the mean size-biased community size is given by
\begin{equation}
\mathbb{E}[S^*]=1+\frac{2\mathbb{E}[D^{\sss{(2)}}]}{3-2\mathbb{E}[D^{\sss{(2)}*}]},
\end{equation}
where $D^{\sss{(2)}*}$ is the size-biased distribution of the triangle degrees. Since $\mathbb{E}[S^*]\geq\mathbb{E}[S]$, Condition~\ref{cond:graph}\ref{cond:S} is satisfied if $\mathbb{E}[D^{\sss{(2)}*}]<\frac32$.

The mean inter-community degree of a community is given by
\begin{equation}
\mathbb{E}[D]=\lim_{n\to\infty}\frac{\sum_{i=1}^n\sum_{v\in G_i}d^{\sss{(1)}}_v}{n}=\lim_{n\to\infty}\frac{\sum_{i=1}^Nd_i^{\sss{(1)}}/N}{n/N}=\mathbb{E}[S]\mathbb{E}[D^{\sss{(1)}}].
\end{equation}
Hence, Conditions~\ref{cond:graph} and~\ref{cond:size} are satisfied if $\mathbb{E}[D^{\sss{(2)}*}]<\frac32$ and $\mathbb{E}[D^{\sss{(1)}}]<\infty$.
When these conditions are satisfied, the condition for emergence of a giant component is
\begin{equation}
\frac{\mathbb{E}[{D^{\sss{(1)}}}^2]\mathbb{E}[S]-\mathbb{E}[D^{\sss{(1)}}]\mathbb{E}[S]}{\mathbb{E}[D^{\sss{(1)}}]\mathbb{E}[S]}=\frac{\mathbb{E}[{D^{\sss{(1)}}}^2]-\mathbb{E}[D^{\sss{(1)}}]}{\mathbb{E}[D^{\sss{(1)}}]}>1.
\end{equation}
Therefore, as long as $\mathbb{E}[D^{\sss{(2)}*}]<\frac32$, the emergence of the giant component only depends on the edge degree distribution.

To apply the results of the hierarchical configuration model, we need the probability $P(H)$ that a randomly chosen community is of type $H$. This probability is not easy to obtain, but it can be approximated using a branching process. The branching process starts at a vertex, and explores the component of triangles. The first generation of the branching process has $Z_0=1$. The first offspring, $Z_1$ is distributed as $2D^{\sss{(2)}}$. All other offspring, $Z_i$ for $i>1$ is distributed as $\sum_{j=1}^{Z_{i-1}}2(D_j^{\sss{(2)}*}-1)$. Here $D_j^{\sss{(2)}*}$ are independent copies, distributed as $D^{\sss{(2)}*}$. In this branching process approximation, cycles of triangles are ignored. The size-biased probability of having a specific community $H$ can be obtained by summing the probabilities of the possible realizations of the branching process when exploring graph $H$. This probability is size-biased, since when starting at an arbitrary vertex, the probability of starting in a larger community is higher. This probability then needs to be transformed to the probability of obtaining graph $H$.

To compute the size of the giant component after percolation from~\eqref{eq:percsize}, $g(H,v,k,\pi)$ is needed for every community shape $H$. This is difficult to obtain, since it largely depends on the shape of the community, and there are infinitely many possible community shapes. Figure~\ref{fig:newmexample} shows an example of why the shape of a community matters. When percolating the left community, the probability that the red vertex is connected to $k$ other vertices is smaller than for the graph on the right. For this reason, we approximate~\eqref{eq:percsize} numerically using the branching process described above.  
In~\cite{newman2009}, Newman gives expressions for the size of the largest percolating cluster. Figure~\ref{fig:percnewm} compares the size of the giant component computed in that way with a numerical approximation of~\eqref{eq:percsize}. We see that indeed the equations from~\cite{newman2009} give the same results for the largest percolating cluster as~\eqref{eq:percsize}.

\begin{figure}[htb]
\begin{minipage}[t]{0.5\linewidth}
	\centering
	\includegraphics[width=\textwidth]{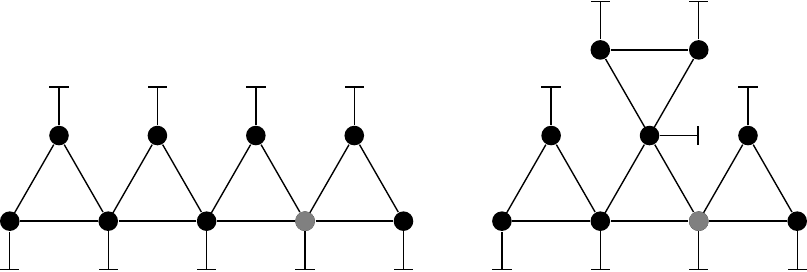}
	\caption{Two possible communities with 4 triangles. In the left community, 6 other nodes can be reached from the gray node within 2 steps, in the right community 8 nodes.}
	\label{fig:newmexample}
\end{minipage}
\hspace{0.2cm}
\begin{minipage}[t]{0.45\linewidth}
	\centering
	\includegraphics[width=\textwidth]{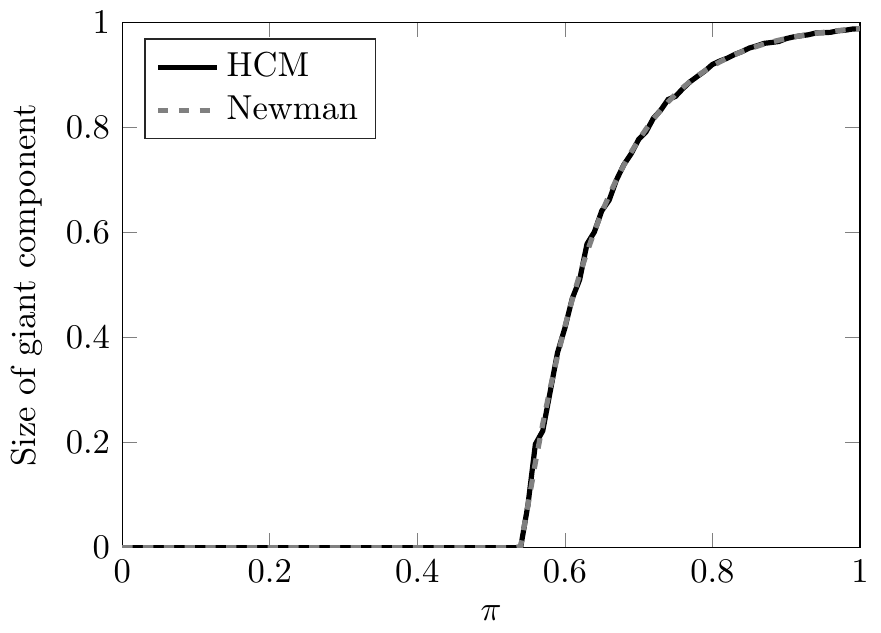}
	\caption{The size of the giant percolating cluster calculated by~\eqref{eq:percsize} (hierarchical configuration model) and from results in~\cite{newman2009} (Newman) agree.} 
	\label{fig:percnewm}
\end{minipage}
\end{figure}

\section{Stylized networks}\label{sec:sty}

In this section, we study two stylized examples of community structures. The first example gives a community type that decreases the critical percolation value compared to a configuration model with the same degree distribution. The second example increases the critical percolation value when compared to a configuration model.

\subsection{A community structure that decreases $\pi_c$}

As an example of a community structure that decreases $\pi_c$, we consider a hierarchical configuration model where with probability $\phi$ a community is given by $H_1$: a path of $L$ vertices, with a half-edge at each end of the path as illustrated in Figure~\ref{fig:line}. With probability $1-\phi$ the community is $H_2$: a vertex with three half-edges. 
The degree distribution of this hierarchical configuration model can be found using Proposition~\ref{prop:deg} and is given by
\begin{equation}\label{eq:linedeg}
p_k=\begin{cases}
	\frac{L\phi}{L\phi+1-\phi} &\text{if } k=2,\\
	\frac{1-\phi}{L\phi+1-\phi} &\text{if } k=3,\\
	0&\text{otherwise}.
\end{cases}
\end{equation}
In this example $\mathbb{E}[D]=2\phi+3(1-\phi)$. Furthermore, $g(H_1,v,2,\pi)=\pi^{L-1}$ for all $v\in H_1$. In $H_2$ there is no percolation inside the community, hence $g(H_2,v,3,\pi)=1$. Equation~\eqref{eq:pic} now gives:
\begin{equation}
\pi_c=\frac{3-\phi}{ 2\phi\pi_c^{L-1}+6(1-\phi)},
\end{equation}
hence $2\phi\pi^{L}_c+6(1-\phi)\pi_c-3+\phi=0$.

Now we let the degree distribution as defined in~\eqref{eq:linedeg} remain the same, while changing the length of the path communities $L$. 
If in the total graph, we want to have a fraction of $a$ vertices of degree 3, then $a=p_3=\frac{1-\phi}{1-\phi+L\phi}$. Hence, $\phi=\frac{1-a}{1-a+La}$. 
In this way, we obtain hierarchical configuration models with the same degree distribution, but with different values of $L$. Figure~\ref{fig:linepi} shows the size of the largest component as calculated by~\eqref{eq:percsize} for $a=1/3$. As $L$ increases, $\pi_c$ decreases. Hence, adding this community structure `helps' the diffusion process. This can be explained by the fact that increasing $L$ decreases the number of line communities. Therefore, more vertex communities will be connected to one another, which decreases the value of $\pi_c$.
Another interesting observation is that the size of the giant component is non-convex in $\pi$. 
These non-convex shapes can be explained intuitively. As the lines get longer, there are fewer and fewer of them, since the degree distribution remains the same. Hence, if $L$ is large, there will only be a few long lines. These lines have $\pi_c\approx1$. Since there are only a few lines, almost all vertices of degree 3 will be paired to one another. The critical value for percolation on a configuration model with only vertices of degree 3 is $0.5$. Hence, for this hierarchical configuration model with $L$ large we will see the vertices of degree 3 appearing in the giant component as $\pi=0.5$, and the vertices in the lines as $\pi=1$. 

\begin{figure}[tb]
\centering
\includegraphics[width=0.3\textwidth]{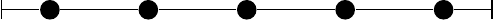}
\caption{A line community with $L=5$}
\label{fig:line}
\end{figure}

\subsection{A community structure that increases $\pi_c$}

As an example of a community structure that inhibits the diffusion process, consider a configuration model with intermediate vertices as introduced in~\cite{hofstad2014}: a configuration model where every edge is replaced by two edges with a vertex in between them. This is equal to a hierarchical configuration model with star-shaped communities as in Figure~\ref{fig:star}: one vertex that is connected to $L$ other vertices. Each of the $L$ other vertices has inter-community degree one. The vertex in the middle is not connected to other communities. We consider a hierarchical configuration model where all communities are stars of the same size. Therefore all star-shaped communities have the same number of outgoing edges, and $\mathbb{E}[D]=L$. 

The degree distribution of this hierarchical configuration model is given by
\begin{equation}\label{eq:stardeg}
p_k=\begin{cases}
\frac{L}{L+1} &\text{if } k=2,\\
\frac{1}{L+1}&\text{if } k=L,\\
0&\text{otherwise}.
\end{cases}
\end{equation}
Under percolation, the connected component of a vertex $v$ at the end point of a star can link to other half-edges only if the edge to the middle vertex is present. Then the number of half-edges to which $v$ is connected is binomially distributed, so that $g(H,v,k,\pi)=\pi{{L-1}\choose{k-1}}\pi^{k-1}(1-\pi)^{L-k}$ for $k\geq2$. Then~\eqref{eq:pic1} gives
\begin{equation}
\begin{aligned}[b]
\mathbb{E}[D_\pi^*]&=\frac{1}{\mathbb{E}[D]}\sum_{H}\sum_k\sum_{v\in V_H}P(H)d_v^{\sss{\text(b)}} kg(H,v,k+1,\pi)\\
&=\pi\sum_{k\geq 1}k\pi^k(1-\pi)^{L-k-1}{L-1\choose k}=(L-1)\pi^2.
\end{aligned}
\end{equation}
Then equation~\eqref{eq:pic} yields $\pi_c=(L-1)^{-1/3}$.

Now we consider a configuration model with the same degree distribution~\eqref{eq:stardeg}. For this configuration model, $\pi_c=\frac{3L}{4L+L^2-3L}=\frac{3}{L+1}$. 
Figure~\ref{fig:starpi} shows the size of the giant component of the hierarchical configuration model compared with a configuration model with the same degree distribution for different values of $L$. This hierarchical configuration has a higher critical percolation value than its corresponding configuration model. Intuitively, this can be explained from the fact that all vertices with a high degree are `hidden' behind vertices of degree 2, whereas in the configuration model, vertices of degree $L$ may be connected to one another. 

Combined with the previous example, we see that adding communities may lead to a higher critical percolation value or a lower one. Furthermore, the size of the giant component may be smaller or larger after adding communities.
\begin{figure}[tb]
\centering
\begin{minipage}[t]{0.45\linewidth}
\centering
\includegraphics[width=\textwidth]{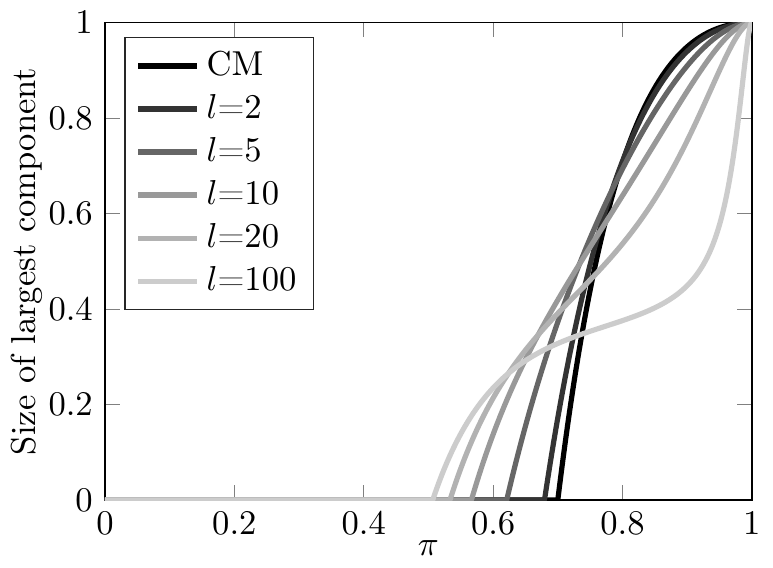}
\caption{Size of giant component against $\pi$ for line communities with different values of $L$.}
\label{fig:linepi}
\end{minipage}
\hspace{0.2cm}
\begin{minipage}[t]{0.45\linewidth}
\centering
\includegraphics[width=\textwidth]{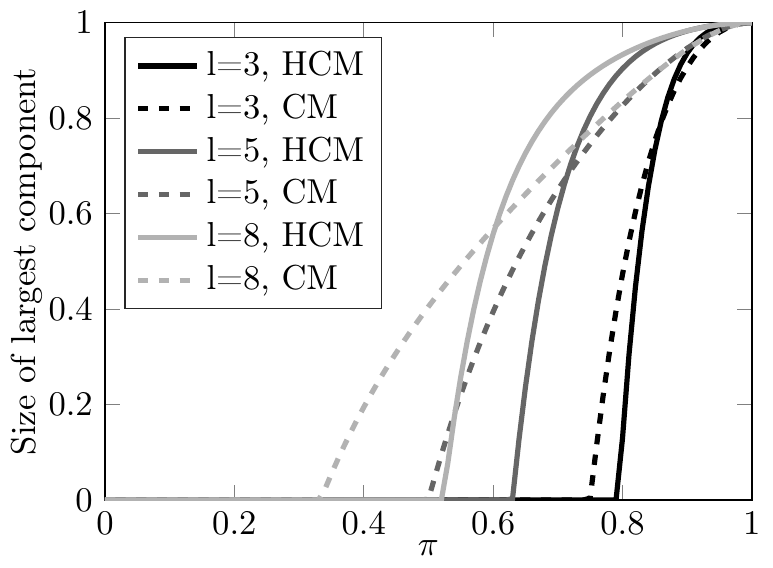}
\caption{Size of giant component against $\pi$ for star communities with different values of $L$.}
\label{fig:starpi}
\end{minipage}
\end{figure}

\section{Conclusions and discussion}\label{sec:disc}

In this paper, we have introduced the hierarchical configuration model, where the macroscopic graph is a configuration model, and on the microscopic level vertices are replaced by communities. We have analytically studied several properties of this random graph model, which led to several interesting insights. For example, the condition for a giant component to emerge in the hierarchical configuration model is completely determined by properties of the macroscopic configuration model. However, the size of the giant component also depends on the community sizes. In contrast, the asymptotic clustering coefficient is entirely defined by the clustering inside the communities. For bond percolation on the hierarchical configuration model, the critical percolation value depends on both the inter-community degree distribution, and the shape of the communities. 
Furthermore, we have shown that if communities are dense with a power-law degree distribution, then the edges between communities follow a power law with an exponent that is one higher than the exponent of  the degree distribution. We have further investigated power-law relations in several real-world networks, and compare these to the power-law relations in our hierarchical configuration model in a companion paper~\cite{stegehuis2015}. These real-world networks do not display this power-law shift, which implies that most communities in real-world networks do not satisfy the intuitive picture of dense communities. In fact, we find a power-law relation between the denseness of the communities and their sizes, so that the large communities are less dense than the smaller communities.

 Finally, we have shown that several existing models incorporating a community structure can be interpreted as a special case of the hierarchical configuration model, which underlines its generality.
Worthwhile extensions of the hierarchical configuration model for future research include directed or weighted counterparts and a version that allows for overlapping communities.

The analysis of percolation on the hierarchical configuration model has shown that the size of the largest percolating cluster and the critical percolation value do not necessarily increase or decrease when adding clustering. It would be interesting to investigate how other characteristics of the graph like degree-degree correlations influence the critical percolation value. 
Another interesting feature of the hierarchical configuration model is its applicability to real-world data sets. In this setting, the hierarchical configuration model creates a graph with the same degree distribution and the same community structure as real-world networks. In~\cite{stegehuis2016} we investigated by simulations how community structures affect the spread of several epidemic processes on real-world networks, including bond percolation and an SIR model, and showed that communities in real-world networks can either inhibit or enforce an epidemic. It would be interesting to study these epidemic processes also analytically.

\newcolumntype{Y}{>{\RaggedRight\arraybackslash}X} 
\begin{table}[htbp]
	\centering
	\begin{tabularx}{400pt}{l Y}
		\toprule
		$H$     & a community type\\
		$P(H)$  & asymptotic probability that a community is of type $H$ \\
		$p_{k,s}$ & asymptotic probability that a community has size $s$ and inter-community degree $k$ \\
		$S$     & asymptotic community size distribution \\
		$D$     & asymptotic community inter-community degree distribution \\
		$d_v^{\sss{(b)}}$    & inter-community degree: the number of edges from vertex $v$ to other communities \\
		$d_v^{\sss{(c)}}$    & intra-community degree: the number of edges from vertex $v$ to community members \\
		$d_v$    & degree of vertex $v$, $d_v=d_v^{\sss{(b)}}+d_v^{\sss{(c)}}$ \\
		$n$     & number of communities \\
		$N$     & number of vertices \\
		$\pi_c$   & critical percolation probability \\
		$g(H,v,k,\pi)$   & probability that vertex $v$ is connected to $k$ edges going out of community $H$ after percolating the edges inside the community with parameter $\pi$ \\
		\bottomrule
	\end{tabularx}%
	\caption{Frequently used symbols}
	\label{tab:addlabel}%
\end{table}%

\paragraph*{Acknowledgement.}
This work is supported by NWO TOP grant 613.001.451.
The work of RvdH is further supported by the NWO VICI grant 639.033.806.  The work of JvL is further supported by an NWO TOP-GO grant and by an ERC Starting Grant.

\bibliographystyle{abbrv}

\end{document}